\documentclass[12pt,A4]{amsart}
\usepackage{amsfonts,amsthm,amsmath}
\usepackage{amssymb}
\usepackage{rotating}
\usepackage{tikz}
\usepackage{verbatim}
\usepackage{amscd}
\usepackage[latin2]{inputenc}
\usepackage{t1enc}
\usepackage[mathscr]{eucal}
\usepackage{indentfirst}
\usepackage{graphicx}
\usepackage{graphics}
\usepackage{pict2e}
\usepackage{mathrsfs}
\usepackage{enumerate}
\usepackage[pagebackref]{hyperref}
\hypersetup{backref, pagebackref, colorlinks=true}
\usepackage{cite}
\usepackage{color}
\usepackage{epic}
\usepackage{hyperref} 
\usepackage{tkz-graph}

\numberwithin{equation}{section}
\def\blue{\textcolor{blue}}
\def\red{\textcolor{red}}
\def\magenta{\textcolor{magenta}}

\theoremstyle{plain}
\newtheorem{theorem}{Theorem}[section]
\newtheorem{lemma}[theorem]{Lemma}
\newtheorem{corollary}[theorem]{Corollary}

\newtheorem{proposition}[theorem]{Proposition}

\theoremstyle{definition}

\newtheorem{example}[theorem]{Example}
\newtheorem{conjecture}[theorem]{Conjecture}
\newtheorem{remark}[theorem]{Remark}
\newtheorem{?}[theorem]{Problem}

\newcommand{\N}{\mathbb{N}}

\newcommand\des{\mathop{\rm des}}

\def\DES{\mathrm{DES}}
\def\st{\mathrm{st}}

\def\max{\mathrm{max}}
\def\min{\mathrm{min}}

\def\des{\mathrm{des}}

\def\S{\mathfrak{S}}

\def\tg{\mathrm{tg}}

\def\ST{\mathrm{ST}}
\def\VID{\mathrm{VID}}
\def\LMA{\mathrm{LMA}}
\def\LMI{\mathrm{LMI}}
\def\RMA{\mathrm{RMA}}
\def\RMI{\mathrm{RMI}}

\def\AVA{\mathrm{AVA}}

\def\DIST{\mathrm{DIST}}
\def\ZERO{\mathrm{ZERO}}
\def\EMA{\mathrm{EMA}}
\def\ROW{\mathrm{ROW}}
\def\last{\mathrm{last}}
\def\cri{\mathrm{cri}}

\def\segm{\mathrm{segment}}
\def\lma{\mathrm{lma}}
\def\rma{\mathrm{rma}}
\def\iasc{\mathrm{iasc}}
\def\IASC{\mathrm{IASC}}

\def\asc{\mathrm{asc}}
\def\ASC{\mathrm{ASC}}
\def\BOT{\mathrm{BOT}}

\def\dist{\operatorname{dist}}
\def\EXPO{\operatorname{EXPO}}
\def\ides{\operatorname{ides}}
\def\max{\operatorname{max}}
\def\mod{\operatorname{mod}}

\def\I{\operatorname{{\bf I}}}

\def\s{\operatorname{{\bf s}}}
\def\st{\operatorname{st}}
\def\segm{\operatorname{segm}}
\def\Alt{\operatorname{Alt}}
\def\IDB{\operatorname{IDB}}

\def\boxit#1{\leavevmode\hbox{\vrule\vtop{\vbox{\kern.33333pt\hrule\kern1pt\hbox{\kern1pt\vbox{#1}\kern1pt}}\kern1pt\hrule}\vrule}}

\usepackage{collectbox}


\begin{document}

\title[Restricted inversion sequences]{Refined restricted inversion sequences}

\author[Z. Lin]{Zhicong Lin}
\address[Zhicong Lin]{Research Center for Mathematics and Interdisciplinary Sciences, Shandong University, Qingdao 266237, P.R. China}
\email{linz@sdu.edu.cn}

\author[D. Kim]{Dongsu Kim}
\address[Dongsu Kim]{Department of Mathematical Sciences, Korea Advanced Institute of Science and Technology, Daejeon, 305-701, Republic of Korea} 
\email{dongsu.kim@kaist.ac.kr}

\date{\today}


\begin{abstract} Recently, the study of patterns in inversion sequences was initiated by Corteel-Martinez-Savage-Weselcouch and Mansour-Shattuck independently.  Motivated by their  works and a double Eulerian equidistribution due to Foata (1977), we investigate several classical statistics on  restricted inversion sequences that are either known or conjectured to be enumerated by  {\em Catalan}, {\em Large Schr\"oder},  {\em Baxter} and {\em Euler} numbers.  One of the two highlights of our results is a fascinating bijection between $000$-avoiding inversion sequences and Simsun permutations, which together with Foata's V- and S-codes, provide a proof of a restriced double Eulerian equdistribution.  The other one is a refinement of a conjecture due to  Martinez and Savage that the cardinality of $\I_n(\geq,\geq,>)$ is the $n$-th Baxter number, which is proved via the so-called {\em obstinate kernel method} developed by Bousquet-M\'elou.
\end{abstract}


\maketitle


\section{Introduction}
For each $n\geq1$, the set of {\em inversion sequences} of length $n$, denoted $\I_n$, is defined by
$\I_n=\{(e_1,e_2,\ldots,e_n): 0\leq e_i<i\}$. It serves as various kind of codings for $\S_n$, the set of
permutations of $[n]:=\{1,2,\ldots,n\}$. By a {\em coding} of $\S_n$, we mean a bijection from $\S_n$ to
$\I_n$. For example, the map $\Theta: \S_n\rightarrow\I_n$ defined for
$\pi=\pi_1\pi_2\ldots\pi_n\in\S_n$ as
$$
\Theta(\pi)=(e_1,e_2,\ldots,e_n),\quad\text{where $e_i:=|\{j<i: \pi_j>\pi_i\}$}|,
$$
is a natural coding of $\S_n$. Clearly, the sum of the entries of $\Theta(\pi)$ equals the number of
{\em inversions} of $\pi$, i.e., the number of pairs $i<j$ such that $\pi_i>\pi_j$. This is the reason why
$\I_n$ is named inversion sequences here.
  
Pattern avoidance in permutations has already been extensively studied in the literature (see the book
by Kitave~\cite{ki}), while the systematic study of patterns in inversion sequences was initiated only
recently in~\cite{cor} and~\cite{mash}. Since both permutations and inversion sequences will be regarded
as words over $\N=\{0,1,\ldots\}$, their patterns can be defined in a unified way as follows. 

For two words $W=w_1w_2\cdots w_n$ and $P=p_1p_2\cdots p_k$ ($k\leq n$) on $\N$, we say that {\em$W$ contains the pattern $P$} if there  exist  some indices $i_1<i_2<\cdots<i_k$ such that the subword $w_{i_1}w_{i_2}\cdots w_{i_k}$ of $W$  is order isomorphic to $P$. Otherwise, $W$ is said to {\em avoid the pattern $P$}. For example, the word $W=32421$ contains the pattern $231$, because the subword $w_2w_3w_5=241$ of $W$ has the same relative order as $231$. However, $W$ is $101$-avoiding. For a set of words $\mathcal{W}$, the set of words in $\mathcal{W}$ avoiding patterns $P_1,\ldots,P_r$ is denoted by $\mathcal{W}(P_1,\ldots,P_r)$. 
One well-known  enumeration result in this area, attributed to MacMahon and Knuth (cf.~\cite{ki}), is that $|\S_n(123)|=C_n=|\S_n(132)|$, where $C_n:=\frac{1}{n+1}{2n\choose n}$ is the {\em $n$-th Catalan number}.

In~\cite{cor,mash}, inversion sequences avoiding patterns of length $3$ are exploited, where a number of familiar combinatorial sequences, such as {\em large Schr\"oder numbers} (denoted $S_n$) and {\em Euler numbers} (denoted by $E_n$), arise. Martinez and Savage~\cite{ms} further considered a generalization of
pattern avoidance to a fixed triple of binary relations $(\rho_1,\rho_2,\rho_3)$. For each triple of relations
$(\rho_1,\rho_2,\rho_3)\in\{<,\,>,\,\leq,\,\geq,\,=,\,\neq,\,-\}^3$, they studied the set $\I_n(\rho_1,\rho_2,\rho_3)$ consisting of those $e\in\I_n$  with no $i<j<k$ such that $e_i\,\rho_1\,e_j$, $e_j\,\rho_2\,e_k$ and
$e_i\,\rho_3\,e_k$. Here the relation $''-''$ on a set $S$ is all of $S\times S$, i.e.,
$x\,''\!\!-''y$ for all $x,y\in S$. For example, $\I_n(<,>,<)=\I_n(021)$ and $\I_n(\geq,-,\geq)=\I_n(000,101,110,100,201,210)$. In Fig.~\ref{patt}, we summarize some of their enumeration results and conjectures, as well as corresponding classical facts in permutation patterns.
\begin{figure}
\setlength {\unitlength} {1mm}
\begin {picture} (200,50) \setlength {\unitlength} {1mm}
\thinlines
\put(3,0){\line(1,0){160}}\put(3,50){\line(1,0){160}}\put(13,0){\line(0,1){50}}

\put(5,44){$C_n$}
\put(18,44){$\S_n(132),\S_n(321)$: classical result~\cite{ki}; $\I_n(\geq,-,\geq)$: conjectured in~\cite{ms}}
\put(3,40){\line(1,0){160}}

\put(5,30){$S_n$}
\put(18,32){$\S_n(2413,3142),\S_n(2413,4213),\S_n(3124,3214)$: classical result~\cite{kre}}
\put(18,24){$\I_n(021)$: proved in~\cite{cor,mash}; $\I_n(\geq,\neq,\geq),\I_n(>,-,\geq),\I_n(\geq,-,>)$: proved in~\cite{ms}}
\put(3,20){\line(1,0){160}}

\put(5,14){$B_n$}
\put(18,14){$\S_n(2\underline{41}3,3\underline{14}2)$: classical result~\cite{chung}; $\I_n(\geq,\geq,>)$: conjectured in \cite{ms}}

\put(3,4){\text{$E_{n+1}$}}
\put(18,4){\text{Simsun permutations of $[n]$: classical result \cite{sun}; $I_n(000)$: proved in~\cite{cor}}}
\put(3,10){\line(1,0){160}}
\end{picture}
\caption{Sets enumerated by  Catalan number $C_n$, Schr\"oder number $S_n$, Baxter number $B_n$ or Euler number $E_{n+1}$.
\label{patt}}
\end {figure}
Based on these results, we will investigate more connections between restricted permutations and inversion sequences by considering several classical statistics that we recall below. 
%
%

For each $\pi\in\S_n$ and each $e\in\I_n$, let
$$
\DES(\pi):=\{i\in[n-1]: \pi_i>\pi_{i+1}\}\quad \text{and}\quad\ASC(e):=\{i\in[n-1]: e_i<e_{i+1}\}
$$ 
be the {\em {\bf\em des}cent set} of $\pi$ and the {\em {\bf\em asc}ent set} of $e$, respectively. Another important property of the coding $\Theta$ is that $\DES(\pi)=\ASC(\Theta(\pi))$ for each $\pi\in\S_n$. Thus,
\begin{equation}\label{des:asc}
\sum_{\pi\in\S_n}t^{\DES(\pi)}=\sum_{e\in\I_n}t^{\ASC(e)},
\end{equation}
where $t^{S}:=\prod_{i\in S}t_i$ for any set $S$ of positive integers.   
Throughout this paper, we use the convention that if ``$\ST$'' is a set-valued statistic, then ``$\st$'' is the corresponding numerical statistic. For example, $\des(\pi)$ is the cardinality of $\DES(\pi)$ for each $\pi$. It is known that $A_n(t):=\sum_{\pi\in\S_n}t^{\des(\pi)}$ is the classical {\em$n$-th Eulerian polynomial}~\cite{fo} and each statistic whose distribution  gives $A_n(t)$ is called a {\em Eulerian statistic}. In view of~\eqref{des:asc}, ``$\asc$'' is a Eulerian statistic on inversion sequences. Let $\dist(e)$  be the {\em number of {\bf\em dist}inct positive entries} of $e$. This statistic was first introduced by Dumont~\cite{du}, who also showed that it is a Eulerian statistic on inversion sequences. Amazingly, Foata~\cite{fo} later invented two different codings of permutations  called {\em V-code} and {\em S-code} to prove the following extension of~\eqref{des:asc}.

\begin{theorem}[Foata~1977]\label{foata}
For each $\pi\in\S_n$ let $\ides(\pi):=\des(\pi^{-1})$ be the number of inverse descents of $\pi$. Then,
\begin{equation}\label{dist:asc}
\sum_{\pi\in\S_n}s^{\ides(\pi)}t^{\DES(\pi)}=\sum_{e\in\I_n}s^{\dist(e)}t^{\ASC(e)}.
\end{equation}
\end{theorem}

Partial results regarding the statistics ``$\asc$'' and ``$\dist$'' on restricted inversion sequences
have already been obtained in~\cite{cor,mash,ms}. In particular, the ascent polynomial 
$$S_n(t):=\sum_{e\in\I_n(021)}t^{\asc(e)}$$ was shown to be {\em palindromic} via a connection with some {\em black-white rooted binary trees} in~\cite{cor}. Inspired by Foata's result, we will consider the joint distribution of ``$\asc$'' and ``$\dist$'' on restricted inversion sequences and prove several  restricted versions of~\eqref{dist:asc}. 
Another interesting statistic for $e\in\I_n$ is the {\em {\bf\em last} entry} of $e$, that we denote $\last(e)$. This statistic turns out to be useful in solving some real root problems in~\cite{sv} and will also lead us to solve two enumeration conjectures.

The rest of this paper deals with refinements of Catalan, Schr\"oder, Baxter and Euler numbers. 
Two highlights of our results are: (i) a bijection between $000$-avoiding inversion sequences and Simsun permutations (see Section~\ref{sec:euler}), which is constructed in the spirit of  Sch\"utzenberger's {\em jeu de taquin}; (ii) a refinement of a conjecture due to Martinez and Savage~\cite{ms} that asserts the cardinality of $\I_n(\geq,\geq,>)$ is the $n$-th {\em Baxter number} (denoted $B_n$), which is proved via Bousquet-M\'elou's {\em obstinate kernel method} (see Section~\ref{sec:bax}).

\section{Catalan numbers}
\label{sec:cat}

Let  $(\rho_1,\rho_2,\rho_3)$ be a relation triple in $\{(\geq,-,\geq),(\geq,-,>),(\geq,\geq,>)\}$.
We introduce the parameter $\cri(e)$ for each $e\in\I_n(\rho_1,\rho_2,\rho_3)$, that we call
the {\em {\bf\em cri}tical value} of $e$, as the minimal integer $c$ such that
$(e_1,\ldots,e_n,c)\in\I_{n+1}(\rho_1,\rho_2,\rho_3)$. Note that ``$\cri$'' depends on the relation triple
$(\rho_1,\rho_2,\rho_3)$.  For example, if we consider $e=(0,1,0,2,2,4)$ as inversion sequence in
$\I_6(\geq,-,>)$, then $\cri(e)=2$. However, $\cri(e)=3$ when $e$ is considered as an inversion sequence
in $\I_6(\geq,-,\geq)$. The reason to introduce ``$\cri$'' is that if $e\in\I_n(\rho_1,\rho_2,\rho_3)$,
then $(e_1,\ldots,e_n,k)$ is in $\I_{n+1}(\rho_1,\rho_2,\rho_3)$ if and only if $\cri(e)\leq k\leq n$.
This parameter will play an important role in our study of the {\em Catalan}, {\em Schr\"oder} and
{\em Baxter triangles} induced by the statistic ``$\last$''. 

As a warm-up, we first show how the critical value can help to prove that the cardinality of  $\I_n(\geq,-,\geq)$ is $C_n$, which was conjectured in~\cite{ms}.
Let us define the refinement
$$
C_{n,k}:=|\{e\in\I_n(\geq,-,\geq): \last(e)=k\}|.
$$
The following recurrence shows that the numbers $C_{n,k}$ generate the  {\em Catalan triangle} that has already been widely studied (see~\href{https://oeis.org/A009766}{OEIS: A009766}).
\begin{proposition}\label{cat:tria}
For $0\leq k\leq n-1$, we have the three-term recurrence 
$$
C_{n,k}=C_{n,k-1}+C_{n-1,k}.
$$
Consequently, $C_{n,k}=\frac{n-k}{n}{n-1+k\choose k}$ are Catalan triangle. 

\end{proposition}
\begin{proof}
Let $\mathfrak{C}_{n,k}:=\{e\in\I_n(\geq,-,\geq): \last(e)=k\}$. We divide $\mathfrak{C}_{n,k}$ into the disjoint union $\mathfrak{A}_{n,k}\cup\mathfrak{B}_{n,k}$, where 
$$\mathfrak{A}_{n,k}=\{e\in\mathfrak{C}_{n,k}: \cri(e_1,e_2,\ldots,e_{n-1})=k\}$$ and $\mathfrak{B}_{n,k}=\mathfrak{C}_{n,k}\setminus\mathfrak{A}_{n,k}$.
Since $\cri(e_1,e_2,\ldots,e_{n-1})\leq k-1$ for $e\in\mathfrak{B}_{n,k}$, the mapping that sends $(e_1,e_2,\ldots,e_{n-1},k)$ to $(e_1,e_2,\ldots,e_{n-1},k-1)$ is a  bijection from $\mathfrak{B}_{n,k}$ to 
$\mathfrak{C}_{n,k-1}$. Therefore, the cardinality of $\mathfrak{B}_{n,k}$ is $C_{n,k-1}$ and so it remains to show that $|\mathfrak{A}_{n,k}|=C_{n-1,k}$. 

Now, we are going to construct a bijection $g:\mathfrak{A}_{n,k}\rightarrow\mathfrak{C}_{n-1,k}$, which will complete the proof of the recurrence for $C_{n,k}$. For each $e\in\mathfrak{A}_{n,k}$,  there is a unique index $i$ such that $e_i=k-1$ and $e_{i+1}\leq k-1$. Define $g(e)$ to be the inversion sequence obtained from $e$ by deleting $e_{n-1}$, if $e_{n-1}=n-2$, or by deleting $e_i$,
otherwise. For example, we have $g(0,1,1,3,2)=(0,1,1,2)$ while $g(0,1,1,2,2)=(0,1,2,2)$.
It is routine to check that $g$ is actually a bijection.
\end{proof}

The beginning of the Catalan triangle $C_{n,k}$ are:
 \begin{eqnarray*}
\begin{array}{ccccccccc}
   1  &      &      &      &      &    \\
   1& 1    &      &      &      &        \\
   1  & 2    & 2    &      &      &          \\
  1  & 3   & 5    &  5   &      &           \\
  1  & 4  & 9   &  14   &  14  &         \\
  1 & 5    & 14  &  28  & 42    & 42.   
\end{array}
\end{eqnarray*}
For each $\pi\in\S_n$, let  $\last(\pi)$ be the {\em{\bf\em last} letter} of $\pi$. Connolly et al.~\cite[Corollary~1]{cgg} showed that $|\{\pi\in\S_n(123):\pi^{-1}(n)=n-k\}|=C_{n,k}$. Since  $\pi\in\S_n(123)\Leftrightarrow\pi^{-1}\in\S_n(123)$, it then follows that  $|\{\pi\in\S_n(123):\last(\pi)=n-k\}|=C_{n,k}$. This is equivalent to the following statement. 

\begin{corollary}For $n\geq1$ and $0\leq k\leq n-1$, we have
$$|\{e\in\I_n(\geq,-,\geq):\last(e)=k\}|=|\{\pi\in\S_n(321):\last(\pi)=k+1\}|.
$$
\end{corollary}

The following stronger equidistribution involving the pair $(\dist,\last)$ is also true, which will be proved by generating function.
\begin{theorem}\label{catalan:asc}
For each $\pi\in\S_n$, let $\asc(\pi):=n-1-\des(\pi)$ be the ascent number of $\pi$. Then,
 \begin{equation}\label{cat:equi}
\sum_{\pi\in\S_n(321)}t^{\asc(\pi)}u^{\last(\pi)}=\sum_{e\in\I_n(\geq,-,\geq)}t^{\dist(e)}u^{\last(e)+1}.
 \end{equation}
\end{theorem}

We first compute the generating function for the left-hand side of~\eqref{cat:equi}. We will apply a simple bijection from Dyck paths to $321$-avoiding permutations. Recall that a {\em Dyck path} of length $n$ is a lattice path in $\N^2$ from $(0,0)$ to
$(n,n)$ using the {\em east step} $(1,0)$ and the {\em north step} $(0,1)$, which does
not pass above the line $y=x$. The {\em height of an east step} in a Dyck path is  
the number of north steps before this east step. It is clear that a Dyck path can be
represented as $d_1d_2\ldots d_n$, where $d_i$ is the height of its $i$-th east step. See Fig.~\ref{ex:varphi} for the Dyck path $000344566$. Denote by $\mathcal{D}_n$ the set of all Dyck paths of length $n$.

For our purpose, we will give a new description of a bijection  $\psi:\mathcal{D}_n\rightarrow\S_n(321)$ that was previously used by Elizalde in~\cite[Section~3]{eli}. For a Dyck path $D=d_1d_2\cdots d_n\in\mathcal{D}_n$, we define $\psi(D)=\pi=\pi_1\pi_2\cdots\pi_n$, where 
\begin{itemize}
\item $\pi_i=d_i+1$ if $d_i\neq d_{i+1}$ or $i=n$; otherwise
\item  if $i$ is the $j$-th smallest integer in $\{k\in[n-1]: d_k=d_{k+1}\}$, then $\pi_i$ is the $j$-th smallest integer in $[n]\setminus\{d_1+1,d_2+1,\ldots,d_n+1\}$.
\end{itemize}
See Fig.~\ref{ex:varphi} for  a  visualization of this bijection for the Dyck path $000344566$.
\begin{figure}[ht]
\begin{center}
\begin{tikzpicture}[scale=.5]
\draw[step=1,color=gray](0,0) grid (9,9);
\draw [thick,color=magenta](0,2)--(1,1)(0,1)--(2,3)(1,3)--(2,2)(4,7)--(5,8)(4,8)--(5,7)(7,8)--(8,9)(7,9)--(8,8);
\draw [thick,color=blue](2,0)--(3,1)(3,0)--(2,1)(3,3)--(4,4)(3,4)--(4,3)(5,4)--(6,5)(6,4)--(5,5)(6,5)--(7,6)(7,5)--(6,6)(8,6)--(9,7)(8,7)--(9,6);

\draw [very thick](0,0)--(3,0)--(3,3)--(4,3)--(4,4)--(6,4)--(6,5)--(7,5)--(7,6)--(9,6)--(9,9);

\draw(.5,-0.5) node{\magenta{$2$}};\draw(1.5,-0.5) node{\magenta{$3$}};
\draw(2.5,-0.5) node{\blue{$1$}};\draw(3.5,-0.5) node{\blue{$4$}};
\draw(4.5,-0.5) node{\magenta{$8$}};
\draw(5.5,-0.5) node{\blue{$5$}};\draw(6.5,-0.5) node{\blue{$6$}};
\draw(7.5,-0.5) node{\magenta{$9$}};
\draw(8.5,-0.5) node{\blue{$7$}};
\end{tikzpicture}
\end{center}
\caption{An example of the bijection $\psi:\mathcal{D}_n\rightarrow\S_n(321)$.}
\label{ex:varphi}
\end{figure}
It is  known that a permutation is 321-avoiding if and only if both the subsequence formed by its excedance values and the one formed by the remaining non-excedance values are increasing. Using this characterization, one can check easily that $\psi$ is in fact a bijection since each $d_i$ of $D$ with $d_i\neq d_{i+1}$ or $i=n$ becomes a non-excedance value of $\psi(D)$ (see the blue crosses in  Fig.~\ref{ex:varphi}). Let us introduce the following two statistics:
\begin{itemize}
\item $\last(D)=d_n$, the height of last east step of $D$;
\item $\segm(D)$, the number of {\bf segm}ents of $D$ with length greater than $1$, where a {\em segment} is a maximal string of consecutive east steps of the same height.
\end{itemize}
Continuing with our Dyck path in Fig.~\ref{ex:varphi}, we have $\last(D)=6$ and $\segm(D)=3$.
The bijection $\psi$ has the following property. 

\begin{lemma}
The bijection  $\psi:\mathcal{D}_n\rightarrow\S_n(321)$ transforms  $(\segm,\last)$ to $(\des,\last-1)$. 
\end{lemma}
Let 
$$
C=C(t,u,x):=\sum_{n\geq1}x^n\sum_{D\in\mathcal{D}_n}t^{\segm(D)}u^{\last(D)}=x+(u+t)x^2+\cdots.
$$
We have the following  expression for $C(t,u,x)$.
\begin{proposition} The function $C(t,u,x)$ is algebraic and has the expression
\begin{equation}\label{gen:C}
C(t,u,x)=\frac{2x(tx-x+1)}{1+2x(ux-tux-1)+\sqrt{1+4ux(ux-tux-1)}}.
\end{equation}
\end{proposition}

\begin{proof}
Let $\mathcal{B}_n$ be the set of Dyck paths in $\mathcal{D}_n$ that begin with an east step follows immediately by a north step. If we introduce 
$$
B(t,u,x):=\sum_{n\geq1}x^n\sum_{D\in\mathcal{B}_n}t^{\segm(D)}u^{\last(D)},
$$
then clearly 
\begin{equation}\label{gen:B}
B(t,u,x)=x+uxC(t,u,x).
\end{equation}

Each Dyck path $D=d_1d_2\cdots d_n\in\mathcal{D}_n\setminus\mathcal{B}_n$ with $k=\min\{i\geq2: d_{i+1}=i\text{ or $i=n$}\}$ can be decomposed uniquely into a pair $(D_1,D_2)$  of Dyck paths, where $D_1=d_2d_3\cdots d_k\in\mathcal{D}_{k-1}$ and $D_2=(d_{k+1}-k)(d_{k+2}-k)\cdots(d_n-k)\in\mathcal{D}_{n-k}$ (possibly empty). This decomposition is reversible and satisfies the following properties: 
$$
\last(D)=\chi(D_2\neq\emptyset)\cdot(k+\last(D_2))+\chi(D_2=\emptyset)\cdot\last(D_1)
$$
and
$$
\segm(D)=\segm(D_1)+\segm(D_2)+\chi(D_1\in\mathcal{B}_{k-1}),
$$
where $\chi(\mathsf{S})$ equals $1$, if the statement $\mathsf{S}$ is true; and $0$, otherwise. Turning this decomposition into generating functions then gives 
\begin{equation}\label{gen:A}
C-B=txB+utxB(t,1,ux)C+x(C-B)+ux(C(t,1,ux)-B(t,1,ux))C.
\end{equation}
Setting $u=1$ in~\eqref{gen:A} and~\eqref{gen:B}, we can solve the two equations to get
$$
C(t,1,x)=\frac{1+2x(x-1-tx)-\sqrt{1-4x(1-x+tx)}}{2x(tx+1-x)}.
$$
Substituting this into~\eqref{gen:A} we get~\eqref{gen:C}. 
\end{proof}

Next we are going  to calculate the generating function for the right-hand side of~\eqref{cat:equi}. Let us define
$$
\tilde{C}=\tilde{C}(t,u,x):=\sum_{n\geq1}x^n\sum_{e\in\I_n(\geq,-,\geq)}t^{n-1-\dist(e)}u^{\last(e)}=x+(u+t)x^2+\cdots.
$$
A decomposition of $(\geq,-,\geq)$-avoiding inversion sequences similar to that of Dyck paths enables us to obtain the following expression for $\tilde{C}(t,u,x)$.
\begin{proposition}
The function $\tilde{C}(t,u,x)$ is algebraic and has the expression
\begin{equation}\label{gen:CC}
\tilde{C}(t,u,x)=\frac{(tx-x+1)\biggl(1-\sqrt{1+4ux(ux-tux-1)}\biggr)}{(tux+1-ux)\biggl(2u-1+\sqrt{1+4ux(ux-tux-1)}\biggr)}.
\end{equation}
\end{proposition}
\begin{proof}
Let $e=(e_1,\ldots,e_n)\in\I_n(\geq,-,\geq)$. We distinguish the following two cases:
\begin{itemize}
\item If $e_n=n-1$, then 
$$
\dist(e)=\dist(e_1,\ldots,e_{n-1})+\chi(n\neq1)\text{ and }\last(e)=n-1.
$$
\item If $k=\max\{i: e_i=i-1\}<n$, then it is straightforward to show that $e$ can be decomposed into two smaller inversion sequences: $f=(e_1,\ldots,e_{k-1},e_{k+1})$ in $\I_k(\geq,-,\geq)$ 
and $g=(e_{k+2}-k, e_{k+3}-k,\ldots,e_n-k)$ in $\I_{n-1-k}(\geq,-,\geq)$ (possibly empty). 
This decomposition is reversible and satisfies the following properties: 
$$
\last(e)=\chi(g=\emptyset)\cdot\last(f)+\chi(g\neq\emptyset)\cdot(\last(g)+k)
$$
and
$$
\dist(e)=\dist(f)+\chi(e_{k+1}\neq k-1)+\chi(g\neq\emptyset)\cdot(\dist(g)+1).
$$
\end{itemize}

For convenience, we introduce 
$$
\tilde{B}=\tilde{B}(t,u,x):=\sum_{n\geq1}x^n\sum_{e\in\I_n(\geq,-,\geq)\atop e_n={n-1}}t^{n-1-\dist(e)}u^{\last(e)}.
$$
It is clear that 
\begin{equation}
\tilde{B}(t,u,x)=x+x\tilde{C}(t,1,ux).
\end{equation}
Translating the above decomposition of $(\geq,-\geq)$-avoiding inversion sequences into generating function yields
\begin{equation}\label{gen:AA}
\tilde{C}-\tilde{B}=tx\tilde{B}+x(\tilde{C}-\tilde{B})+tx\tilde{B}(t,1,ux)\tilde{C}+x(\tilde{C}(t,1,ux)-\tilde{B}(t,1,ux))\tilde{C}.
\end{equation}
Similar to~\eqref{gen:A}, we can solve~\eqref{gen:AA} to obtain~\eqref{gen:CC}.
\end{proof}

\begin{proof}[Proof of Theorem~\ref{catalan:asc}] To check the right-hand side of~\eqref{gen:C} equals that of~\eqref{gen:CC} is routine (by Maple) and we conclude that $C(t,u,x)=\tilde{C}(t,u,x)$, which is equivalent to equidistribution~\eqref{cat:equi}.
\end{proof}
\section{Schr\"oder numbers}
\subsection{A new Schr\"oder triangle}
\begin{theorem}\label{thm:sch}
For $n\geq1$ and $0\leq k\leq n-1$, we have
\begin{equation}\label{sch:tri}
|\{e\in\I_n(\geq,-,>): \last(e)=k\}|=|\{e\in\I_n(021): \last(e)\equiv k+1 (\mod\,n)\}|.
\end{equation}
\end{theorem}

Note that this result is obviously true for $k=n-1,n-2,n-3$. Let us define the {\em Schr\"oder triangle} 
$S_{n,k}:=|\{e\in\I_n(\geq,-,>): \last(e)=k\}|$.
The first values of $S_{n,k}$ are:
\begin{eqnarray*}
\begin{array}{ccccccccccc}
   1  &      &      &      &      &      \\
   1& 1    &      &      &      &     \\
   2  & 2    & 2    &      &      &   \\
  4  & 6   & 6    &  6   &      &     \\
  8  & 16   & 22   &  22   &  22   &  \\
  16 & 40  & 68  & 90 &  90 &  90.
\end{array}
\end{eqnarray*}
We have the following simple recurrence for $S_{n,k}$.
\begin{lemma}\label{lem:sch}
For $0\leq k\leq n-3$, we have the  four-term recurrence 
$$
S_{n,k}=S_{n,k-1}+2S_{n-1,k}-S_{n-1,k-1}.
$$
\end{lemma}
\begin{proof}
As in the Catalan case, we divide the set $\mathcal{S}_{n,k}:=\{\I_n(\geq,-,>):\last(e)=k\}$ into
the disjoint union $\mathcal{A}_{n,k}\cup\mathcal{B}_{n,k}$, where 
$$\mathcal{A}_{n,k}:=\{e\in\mathcal{S}_{n,k}: \cri(e_1,\ldots,e_{n-1})=k\}$$
and 
$\mathcal{B}_{n,k}=\mathcal{S}_{n,k}\setminus \mathcal{A}_{n,k}$. Clearly, there is a natural bijection
from $\mathcal{B}_{n,k}$ to $\mathcal{S}_{n,k-1}$, which maps $(e_1,\ldots,e_{n-1},k)$ to
$(e_1,\ldots,e_{n-1},k-1)$. Therefore, the cardinality of $\mathcal{B}_{n,k}$ is $S_{n,k-1}$ and so it
remains to show $|\mathcal{A}_{n,k}|=2S_{n-1,k}-S_{n-1,k-1}$, assuming $k\leq n-3$. 
To do this, we further divide $\mathcal{A}_{n,k}$ into the disjoint union
$\mathcal{C}_{n,k}\cup\mathcal{D}_{n,k}$, where 
$$
\mathcal{C}_{n,k}:=\{e\in\mathcal{A}_{n,k}: e_{n-1}=n-2,\ \cri(e_1,\ldots,e_{n-2})=k\}\\
$$
and $\mathcal{D}_{n,k}=\mathcal{A}_{n,k}\setminus\mathcal{C}_{n,k}$. Obviously, we have
$\{(e_1,\ldots,e_{n-2},e_n): e\in \mathcal{C}_{n,k}\}=\mathcal{A}_{n-1,k}$.
Thus,
$$|\mathcal{C}_{n,k}|=|\mathcal{A}_{n-1,k}|=|\mathcal{S}_{n-1,k}|-|\mathcal{B}_{n-1,k}|=S_{n-1,k}-S_{n-1,k-1},$$
which will end the proof once we can define a bijection from  $\mathcal{D}_{n,k}$ to $\mathcal{S}_{n-1,k}$. 
 
For each $e\in\mathcal{D}_{n,k}$, if $e_i$ is the left-most entry that equals $\cri(e)=k$,
then the entries $e_i,e_{i+1},\ldots,e_{n-1}$ of $e$ must satisfy: 
\begin{itemize}
\item[(i)] $e_i=k$ and $e_{i+1}\leq k$;
\item[(ii)] $k\leq e_{i+2}\leq e_{i+3}\leq\cdots\leq e_{n-1}$, where the inequalities  after the entries greater
than $k$ are strict.
\end{itemize}
Now removing the right-most entry $e_j$, such that $e_j=k$ and $i\leq j\leq n-1$, from $e$ results in an inversion sequence in $\mathcal{S}_{n-1,k}$ (since $e_{n-1}\leq n-3$) that we denote $f(e)$. For example, we have $f(0,1,2,0,2,2)=(0,1,2,0,2)$, $f(0,1,0,2,2,2)=(0,1,0,2,2)$ and $f(0,1,2,1,3,2)=(0,1,1,3,2)$. We claim that the map $f:\mathcal{D}_{n,k}\rightarrow\mathcal{S}_{n-1,k}$  is a bijection. 
\end{proof}

\begin{proof}[Proof of Theorem~\ref{thm:sch}]
It is not hard to show that the right-hand side of~\eqref{sch:tri} satisfies the same recurrence relation as $S_{n,k}$, which completes the proof of the theorem.
\end{proof}

One may ask if there is any other interpretation of $S_{n,k}$ in terms of pattern-avoiding permutations. The following conjecture will answer this question completely, if true.

\begin{conjecture}\label{schroder:asc}
Let $(\sigma,\pi)$ be a pair of patterns of length $4$. Then,
$$
S_{n,k}=|\{\pi\in\S_n(\sigma,\pi): \last(\pi)-1=k\}|
$$
 for any $0\leq k<n$ if and only if $(\sigma,\pi)$ is one of the following nine pairs:
 \begin{align*}
 &\qquad\qquad\qquad\qquad\,\,(4321,3421),(3241,2341),(2431,2341),\nonumber\\
 &(4231,3241),(4231,2431),(4231,3421),(2431,3241),(3421,2431),(3421,3241).
 \end{align*}
 Moreover, if $(\sigma,\pi)$ is one of the last six  pairs (i.e. these in second line above), then
 $$
\sum_{\pi\in\S_n(\sigma,\pi)}t^{\asc(\pi)}u^{\last(\pi)}=\sum_{e\in\I_n(\geq,-,>)}t^{\dist(e)}u^{\last(e)+1}.
 $$
\end{conjecture}


\subsection{Double Eulerian equidistributions}
\subsubsection{Statistics}
Let $\pi\in\S_n$ be a permutation. The {\em {\bf\em v}alues of {\bf\em i}nverse {\bf\em d}escents} of $\pi$ is  
$$\VID(\pi):=\{2\leq i\leq n:\pi_i+1 \,\,\text{appears to the left of}\,\, \pi_i\},$$
which  is an important set-valued extension of ``$\ides$''. The {\em positions of {\bf\em l}eft-to-right {\bf\em ma}xima} of $\pi$ is $\LMA(\pi):=\{i\in[n]:\pi_i>\pi_j\,\, \text{for all $1\leq j<i$}\}$. Similarly, we can define the {\em positions of {\bf\em l}eft-to-right {\bf\em mi}xima} $\LMI(\pi)$, the {\em positions of {\bf\em r}ight-to-left {\bf\em ma}xima} $\RMA(\pi)$ and  the {\em positions of {\bf\em r}ight-to-left {\bf\em mi}nima} $\RMI(\pi)$ of $\pi$.

Let $e\in\I_n$ be an inversion sequence. The positions of the {\em last occurrence of {\bf\em dist}inct positive entries} of $e$ is $\DIST(e):=\{2\leq i\leq n: e_i\neq0\,\,\text{and $e_i\neq e_j$ for all $j>i$}\}$. The {\em  positions of {\bf\em zero}s in $e$} is $\ZERO(e):=\{i\in[n]: e_i=0\}$. The {\em positions of the {\bf\em e}ntries of $e$ that achieve {\bf\em ma}ximum} is $\EMA(e):=\{i\in[n]: e_i=i-1\}$ and the {\em positions of {\bf\em r}ight-to-left {\bf\em mi}nima} of~$e$ is $\RMI(e):=\{i\in[n]: e_i<e_j\,\, \text{for all $ j>i$}\}$.

\subsubsection{A sextuple equidistribution}
\label{sec:sex}
Note that an inversion sequence avoids $021$ if and only if its positive entries are weakly increasing. Permutations avoiding the patterns $2413$ and $3142$ are called {\em separable permutations} (cf.~\cite{ki}). Separable permutations and $021$-avoiding inversion sequences are all  enumerated by the large Schr\"oder numbers. Moreover, the work by
Corteel et al.~\cite{cor}  and Fu et al.~\cite{flz} show that 
$$
\sum_{e\in\I_n(021)}t^{\asc(e)}=\sum_{\pi\in\S_n(2413,3142)}t^{\des(\pi)}.
$$
It is this observation that inspires us to find the following sextuple equidistribution, which is an extension of a restricted version of Theorem~\ref{foata}.
\begin{theorem}\label{thm:sex} 
There exists a bijection $\Psi:\I_n(021) \rightarrow \S_n(2413,4213)$ such that
$$
(\DIST,\ASC,\ZERO,\EMA,\RMI,\EXPO)e=(\VID,\DES,\LMA,\LMI,\RMA,\RMI)\Psi(e)
$$
for each $e\in\I_n(021)$, where $\EXPO$ is the {\em {\bf\em expo}sed positions} of e. 
\end{theorem}
The details of constructing  $\Psi$, as well as its two interesting applications, is provided in~\cite{kl}. In the rest of this section, we will show two more restricted versions of Theorem~\ref{foata}.

\subsubsection{Two more equidistributions}
Based on calculations, Martinez and Savage~\cite{ms} suspected that 
$$
\sum_{e\in\I_n(021)}t^{\asc(e)}=\sum_{e\in\I_n(\geq,\neq,\geq)}t^{\asc(e)}=\sum_{e\in\I_n(>,-,\geq)}s^{\asc(e)}.
$$
This follows from Theorem~\ref{thm:sex}, the palindromicity of $S_n(t)$ and two more multivariate equidistributions (Theorems~\ref{equi:1} and~\ref{equi:2}) stated below. 

First we introduce a set-valued extension of ``$\dist$'' different from ``$\DIST$'': 
$$
\ROW(e):=\{e_1,e_2,\ldots,e_n\}\setminus\{0\},\ \text{for each $e\in\I_n$}.
$$
\begin{theorem}\label{equi:1}For $n\geq1$, we have 
$$
\sum_{e\in\I_n(\geq,\neq,\geq)}s^{\ROW(e)}t^{\ASC(e)}u^{\last(e)}=\sum_{e\in\I_n(>,-,\geq)}s^{\ROW(e)}t^{\ASC(e)}u^{\last(e)}.
$$
\end{theorem}
\begin{proof}
We will construct a bijection  $\mathcal{R}:\I_n(\geq,\neq,\geq)\rightarrow\I_n(>,-,\geq)$, which preserves the triple statistics $(\ROW,\ASC,\last)$. Notice that $\I_n(\geq,\neq,\geq)=\I_n(\blue{110},101,201,210)$, while
$\I_n(>,-,\geq)=\I_n(\blue{100},101,201,210)$. The idea is to replace iteratively occurrences
of pattern $100$ in an inversion sequence in $\I_n(\geq,\neq,\geq)\setminus\I_n(>,-,\geq)$ with those of patterns $110$.

Our $\mathcal{R}$ when restricted to    $\I_n(\geq,\neq,\geq)\cap\I_n(>,-,\geq)$ is simply identity. So we only need to define the mapping $\mathcal{R}$ from $\I_n(\geq,\neq,\geq)\setminus\I_n(>,-,\geq)$ to $\I_n(>,-,\geq)\setminus\I_n(\geq,\neq,\geq)$. Let $e=(e_1,\ldots,e_n)\in\I_n(\geq,\neq,\geq)\setminus\I_n(>,-,\geq)$. Clearly,  $e$ must contain the pattern $100$. Find the (unique) greatest entry $e_i$ such that there exists $i<j<k$ and 
$e_i>e_j=e_k$.    It is routine to check that $e$ has the structure 
$$
e=(e_1,\ldots,e_{i-1}, \blue{e_i, e_{i+1},\ldots, e_{j'}},\red{e_{j'+1},\ldots},\blue{e_{k'}},\red{e_{k'+1},\ldots,e_n}),
$$
where $i<j'\leq k'\leq n$ and 
\begin{itemize}
\item $\max\{e_1,\ldots,e_{i-1}\}<e_i>e_{i+1}=\cdots=e_{j'}=e_{k'}$, 
\item $\min\{e_{j'+1},\ldots,e_{k'-1}, e_{k'+1},\ldots,e_n\}>e_i$.
\end{itemize}
Replace the entries $e_{i+1},\ldots,e_{j'}$ of $e$ by $(j'-i)$ copies of $e_i$ and keep other entries unchanged.  The resulting inversion sequence, that we denote $e'$, avoids all patterns inside $\{101,201,210\}$ but contains the patter $110$. If $e'$ avoids the pattern $100$, then define $\mathcal{R}(e)=e'$. Otherwise, repeat the same operation on $e'$ as what we have done on $e$ until we get an inversion sequences inside $\I_n(>,-,\geq)\setminus\I_n(\geq,\neq,\geq)$ which is defined to be $\mathcal{R}(e)$. For example, if $e=(0,1,0,\blue{2,0,0,}3,\blue{0},4)$, then we have the following two steps of replacements:
$$
e=(0,1,0,\blue{2,0,0,}3,\blue{0},4)\rightarrow(0,\blue{1,0},2,2,2,3,\blue{0},4)\rightarrow (0,1,1,2,2,2,3,0,4)=\mathcal{R}(e).
$$
Since each step of replacement is reversible, the mapping $\mathcal{R}$ is bijective. It is obvious that each step of replacement preserves the triple statistics $(\ROW,\ASC,\last)$, and so does $\mathcal{R}$, which completes the proof. 
\end{proof}

Recently, Baril and Vajnovszki~\cite{bv} constructed a new coding $b:\S_n\rightarrow\I_n$ satisfying 
$$
(\VID,\DES,\LMA,\LMI,\RMA)\pi=(\DIST,\ASC,\ZERO,\EMA,\RMI)b(\pi)
$$
for each $\pi\in\S_n$. Their coding can be applied to give an interpretation of $S_n(t)$ in terms of ascent polynomial on $(\geq,\neq,\geq)$-inversion sequences. For this purpose, we will review briefly the construction of $b$ next.  

An {\em integer interval} (or {\em interval} for short) $[m,n]$, $m<n$, is the set $\{x\in\N: m\leq x\leq n\}$. A {\em labelled interval} is a pair $(I,\ell)$, where $I$ is an interval and $\ell$ is an integer.  For a given permutation $\pi=\pi_1\cdots\pi_n\in\S_n$ and an integer $i$, $0\leq i<n$, define the {\em$i$-th slice} of $\pi$, denoted $U_i(\pi)$, to be a sequence of labelled intervals  constructed recursively by the following process. Set $U_0(\pi)=([0,n],0)$. For $i\geq1$, if 
$$U_{i-1}(\pi)=(I_1,\ell_1),(I_2,\ell_2),\ldots,(I_k,\ell_k)$$ is the $(i-1)$-th slide of $\pi$ and $v$, $1\leq v\leq k$, is the index such that $\pi_{i}\in I_v$, then $U_{i}(\pi)$ is constructed according to the following four possible cases:  
\begin{itemize}
\item If $\min(I_v)<\pi_i=\max(I_v)$, then $U_{i}(\pi)$ equals
$$
(I_1,\ell_1),\ldots,(I_{v-1},\ell_{v-1}),\blue{(J,\ell_{v+1}),(I_{v+1},\ell_{v+2}),\ldots,(I_{k-1},\ell_k),(I_k,\ell_k+1)},
$$
where $J=[\min(I_v),\pi_i-1]$. 
\item If $\min(I_v)<\pi_i<\max(I_v)$, then $U_{i}(\pi)$ equals
$$
(I_1,\ell_1),\ldots,(I_{v-1},\ell_{v-1}),\blue{(H,\ell_{v}),(J,\ell_{v+1}),(I_{v+1},\ell_{v+2}),\ldots,(I_{k-1},\ell_k),(I_k,\ell_k+1)},
$$
where $H=[\pi_{i}+1,\max(I_v)]$ and $J=[\min(I_v),\pi_i-1]$. 
\item If $\min(I_v)=\pi_i<\max(I_v)$, then $U_{i}(\pi)$ equals
$$
(I_1,\ell_1),\ldots,(I_{v-1},\ell_{v-1}),\blue{(H,\ell_{v})},(I_{v+1},\ell_{v+1}),\ldots,(I_{k-1},\ell_{k-1}),\blue{(I_k,\ell_k+1)},
$$
where $H=[\pi_{i}+1,\max(I_v)]$. 
\item If $\min(I_v)=\pi_i=\max(I_v)$, then $U_{i}(\pi)$ equals
$$
(I_1,\ell_1),\ldots,(I_{v-1},\ell_{v-1}),(I_{v+1},\ell_{v+1}),\ldots,(I_{k-1},\ell_{k-1}),\blue{(I_k,\ell_k+1)}.
$$
\end{itemize}
Now, let $b(\pi)=(b_1,b_2,\cdots, b_n)\in\I_n$, where for each $i$, $1\leq i\leq n$, $b_i=\ell_v$ if $v$ is such that $(I_v,\ell_v)$ is a labelled interval in the  $(i-1)$-th slice of $\pi$ with $\pi_i\in I_v$. 
\begin{example}
For $\pi=341652\in\S_6$, we compute 
\begin{align*}
&U_0(\pi)=([0,6],0);\\
&U_1(\pi)=([4,6],0),([0,2],1);\\
&U_2(\pi)=([5,6],0),([0,2],2);\\
&U_3(\pi)=([5,6],0),([2,2],2),([0,0],3);\\
&U_4(\pi)=([5,5],2),([2,2],3),([0,0],4);\\
&U_5(\pi)=([2,2],3),([0,0],5).
\end{align*}
Therefore, we get $b(\pi)=(0,0,2,0,2,3)$. 
\end{example}

We say that an interval $I$ is {\em lower} than another interval $J$ if $\max(I)<\min(J)$. 
From the construction of $b$, it is easily checked by induction that  the following properties hold. 
\begin{lemma}\label{lem:dec}
Let $\pi\in\S_n$  and $0\leq i<n$. If $U_{i}(\pi)=(I_1,\ell_1),(I_2,\ell_2),\ldots,(I_k,\ell_k)$, 
then the interval $I_1,I_2,\ldots,I_k$ are in decreasing order, while their labelings $\ell_1,\ell_2,\ldots,\ell_k$ are strictly increasing. Moreover, the labelings $\ell_1,\ell_2,\ldots,\ell_{k-1}$ must appear as entries of $b(\pi)$ after its $i$-th entry. 
\end{lemma}

\begin{theorem}\label{equi:2}The coding $b$ restricts to a bijection from $\S_n(3124,3142)$ to $\I_n(\geq,\neq,\geq)$. In particular,  
$$
\sum_{\pi\in\S_n(3142,3124)}s^{\VID(\pi)}t^{\DES(\pi)}=\sum_{e\in\I_n(\geq,\neq,\geq)}s^{\DIST(e)}t^{\ASC(e)}.
$$
\end{theorem}
\begin{proof}
Since $\S_n(3124,3142)$ and $\I_n(\geq,\neq,\geq)$ have the same cardinality (see Fig.~\ref{patt}), we only need to show that if $\pi\notin\S_n(3124,3142)$, then $b(\pi)\notin\I_n(\geq,\neq,\geq)$. 

Suppose $\pi=\pi_1\pi_2\cdots\pi_n$ is a permutation contains at least a pattern $3124$ or $3142$, then there exists $i,j,k,l$, $1\leq i<j<k<l\leq n$, such that $\pi_i\pi_j\pi_k\pi_l$ is order isomorphic  to $3124$ or $3142$.  Let $b(\pi)=(b_1,b_2,\cdots, b_n)\in\I_n$. As $\pi_i$ plays the role of $3$ in $\pi_i\pi_j\pi_k\pi_l$ and $U_i(\pi)$ is obtained from $U_{i-1}(\pi)$ by removing $\pi_i$, we have that $\max\{\pi_k,\pi_l\}$ lies in an interval different with and lower than the interval contains $\pi_j$ in $U_i(\pi)$. Also, $\pi_k$ and $\pi_l$ lie in different intervals of $U_i(\pi)$. Therefore, in view of Lemma~\ref{lem:dec}, in $U_j(\pi)$ the two intervals contain $\pi_k$ or $\pi_l$, which are different with the interval contains $0$,  have labelings  smaller or equals to $b_j$. Since these two labelings are not the labeling of the interval contains $0$, they will appear  in $b(\pi)$ after $b_j$, which together with $b_j$ form a $(\geq,\neq,\geq)$-pattern in $b(\pi)$. This completes the proof of the theorem. 
\end{proof}

\begin{remark}
Interestingly, we can  also  show that $b$ restricts to a bijection between $\S_n(2413,4213)$ and $\I_n(021)$. Due to cardinality reason, we only need to show that if $\pi\notin\S_n(2413,4213)$, then $b(\pi)\notin\I_n(021)$. To see this, suppose that $\pi_i\pi_j\pi_k\pi_l$, $1\leq i<j<k<l\leq n$, is a $2413$ or $4213$ pattern of $\pi=\pi_1\pi_2\cdots\pi_n$ and $b(\pi)=(b_1,b_2,\cdots, b_n)$. Since $\pi_i$ and $\pi_j$, which play the roles of $2$ and $4$ in $\pi_i\pi_j\pi_k\pi_l$, have been removed from intervals of $U_{j-1}(\pi)$, the two different intervals contain $\pi_k$ or $\pi_l$ have positive labelings in $U_{j-1}(\pi)$. So in $U_k(\pi)$, the labeling of the interval contains $\pi_l$ must positive and smaller than $b_k$. In view of Lemma~\ref{lem:dec},  this labeling must appear after $b_k$, which together with $b_k$ and $b_1$ form a $021$-pattern of $b(\pi)$. This shows the restricted mapping $b: \S_n(2413,4213)\rightarrow\I_n(021)$ is a bijection. 

Note that this restricted $b$ does not transform ``$\RMI$'' to ``$\EXPO$'', while  our bijection $\Psi$ in Theorem~\ref{thm:sex} does. 
\end{remark}

\section{Baxter numbers}
\label{sec:bax}
 A permutation avoiding both {\em vincular patterns} (see~\cite{ki} for the definition) $2\underline{41}3$ and $3\underline{14}2$ is called a {\em Baxter permutation}. It is a result of Chung et al.~\cite{chung} that
$$
B_n=|\S_n(2\underline{41}3,3\underline{14}2)|=\frac{1}{{n+1\choose1}{n+1\choose2}}\sum_{k=0}^{n-1}{n+1\choose k}{n+1\choose k+1}{n+1\choose k+2}.
$$
The number $B_n$ is known as the $n$-th {\em Baxter number}. Martinez and Savage~\cite{ms} conjectured that $|\I_n(\geq,\geq,>)|=B_n$, which can be refined as follows.

\begin{theorem}\label{thm:baxter}
For $n\geq1$, we have the equidistribution
\begin{equation}\label{bax:equi}
\sum_{e\in\I_n(\geq,\geq,>)}u^{n+1-\cri(e)}=\sum_{\pi\in\S_n(2\underline{41}3,3\underline{14}2)}u^{\lma(\pi)+\rma(\pi)}.
\end{equation}
\end{theorem}

In view of Theorem~\ref{thm:baxter},  the $(\geq,\geq,>)$-avoiding inversion sequences will be named {\em Baxter inversion sequences}, which are the only pattern avoiding inversion sequences known to be counted by Baxter numbers. For $0\leq k\leq n-1$, define the {\em Baxter triangle} as 
$$
B_{n,k}:=|\{e\in\I_n(\geq,\geq,>): \last(e)=k\}|.
$$
The first values of the {\em Baxter triangle} $B_{n,k}$ are:
\begin{eqnarray*}
\begin{array}{ccccccccccc}
   1  &      &      &      &      &       &     \\
   1& 1    &      &      &      &       &      \\
   2  & 2    & 2    &      &      &       &      \\
  4  & 6   & 6    &  6   &      &       &      \\
  8  & 18   & 22   &  22   &  22   &       &     \\
  16 & 50  & 80  &  92  &  92  &  92    &     \\
  32 & 130  & 268  &  378 &  422  &  422   & 422.
\end{array}
\end{eqnarray*}
Note that the second column appears as sequence \href{https://oeis.org/A048495}{OEIS: A048495}.

\begin{corollary}
 For $0\leq k\leq n-1$, we have
$$
B_{n,k}=|\{\pi\in\S_{n-1}(2\underline{41}3,3\underline{14}2): \lma(\pi)+\rma(\pi)\geq n-k\}|.
$$
\end{corollary}
%

The rest of this section is devoted to a proof of Theorem~\ref{thm:baxter}.
For each $e\in\I_n(\geq,\geq,>)$, introduce the {\em parameters} $(p,q)$ of $e$, where 
$$p=\max(e)+1-\cri(e) \quad{ and } \quad q=n-\max(e)
$$ with $\max(e):=\max\{e_1,\ldots,e_n\}$. 
For example, if $e=(0,1,0,2,2,4)\in\I_6(\geq,\geq,>)$, then $\max(e)=4$ and $\cri(e)=2$, and so  the parameters of $e$  is $(3,2)$.
After a careful discussion we can obtain the following new rewriting rule.

\begin{lemma}\label{lem:baxter}
Let $e\in\I_n(\geq,\geq,>)$ be a Baxter inversion sequence with parameters $(p,q)$. Exactly $p+q$ Baxter inversion sequences in $\I_{n+1}(\geq,\geq,>)$ when removing their last entries will become $e$, and their  parameters are respectively:
\begin{align*}
&(p-1,q+1), (p-2,q+1), \ldots, (1,q+1), \\
&(1,q+1), (p+1,q), (p+2,q-1),\ldots, (p+q,1).
\end{align*}
The order in which the parameters are listed corresponds to the inversion sequences with last entries from $c$ to $n$, where $c=n+1-(p+q)$.
\end{lemma}

\begin{proof}
It is clear from the definition of  critical value of $e$ that $f=(e_1,\ldots,e_n,b)$ is a Baxter inversion sequence if and only if $\cri(e)\leq b\leq n$. We distinguish three cases:
\begin{itemize}
\item If $\cri(e)\leq b< \max(e)$, then $\cri(f)=b+1$ and $\max(f)=\max(e)$. These Baxter inversion sequences contribute the paramaters $(p-1,q+1), (p-2,q+1), \ldots, (1,q+1)$.
\item If $b=\max(e)$, then $\cri(f)=\max(f)=\max(e)$. This Baxter inversion sequence contributes the paramater $(1,q+1)$.
\item If $\max(e)< b\leq n$, then $\cri(f)=\cri(e)$ and $\max(f)=b$. These Baxter inversion sequences contribute the paramaters $(p+1,q), (p+2,q-1),\ldots, (p+q,1)$.
\end{itemize}
Summing over all the above cases give the desired rewriting rule for Baxter inversion sequences. 
\end{proof}

According to the above rewriting rule, we can construct a {\em generating tree} (actually an infinite rooted tree) for Baxter inversion sequences by representing each element as its parameters like this: the root is $(1,1)$ and the children of a vertex labelled $(p,q)$ are those that generated according to the  rewriting rule in Lemma~\ref{lem:baxter}. See Fig.~\ref{tree:bax} for the first few levels of this generating tree. Note that the number of vertices in the $n$-th level of this tree is the cardinality of $\I_n(\geq,\geq,>)$.

\begin{figure}
\setlength{\unitlength}{1mm}
\begin{picture}(106,32)\setlength{\unitlength}{1mm}
\thinlines
\put(50,27){$(1,1)$}
\put(54,25.5){\line(-5,-1){25}}\put(55,25.5){\line(5,-1){25}}
\put(24,17){\small{$(1,2)$}}\put(76,17){\small{$(2,1)$}}

\put(27.5,15.5){\line(-5,-1){25}}\put(28.5,15.5){\line(-1,-2){2.5}}\put(29.5,15.5){\line(4,-1){20}}
\put(-1.5,7){\footnotesize{$(1,3)$}}\put(22.5,7){\footnotesize{$(2,2)$}}\put(46,7){\footnotesize{$(3,1)$}}

\put(1.5,5.5){\line(-3,-1){15}}
\put(1.5,5.5){\line(-1,-1){5}}
\put(1.5,5.5){\line(0,-1){5}}
\put(1.5,5.5){\line(2,-1){10}}

\put(26,5.5){\line(-2,-1){10}}
\put(26,5.5){\line(-1,-1){5}}
\put(26,5.5){\line(1,-1){5}}
\put(26,5.5){\line(2,-1){10}}

\put(49,5.5){\line(-2,-1){10}}
\put(49,5.5){\line(-1,-2){2.5}}
\put(49,5.5){\line(1,-1){5}}
\put(49,5.5){\line(2,-1){10}}

\put(80,15.5){\line(-2,-1){10}}\put(81,15.5){\line(1,-1){5}}\put(82,15.5){\line(5,-1){25}}
\put(66,7){\footnotesize{$(1,2)$}}\put(83,7){\footnotesize{$(1,2)$}}\put(104,7){\footnotesize{$(3,1)$}}

\put(69.5,5.5){\line(-1,-1){5}}
\put(69.5,5.5){\line(0,-1){5}}
\put(69.5,5.5){\line(1,-1){5}}

\put(86.5,5.5){\line(-1,-1){5}}
\put(86.5,5.5){\line(0,-1){5}}
\put(86.5,5.5){\line(1,-1){5}}

\put(108,5.5){\line(-2,-1){10}}
\put(108,5.5){\line(-1,-2){2.5}}
\put(108,5.5){\line(1,-1){5}}
\put(108,5.5){\line(2,-1){10}}

\end{picture}
\caption{The generating tree for Baxter inversion sequences.}
\label{tree:bax}
\end{figure}

Define the formal power series $F(t;u,v)=F(u,v):=\sum_{p,q\geq1}F_{p,q}(t)u^pv^q$, where $F_{p,q}(t)$ is the size generating function for Baxter inversion sequences with parameters $(p,q)$. We can turn the above lemma into a functional equation as follows. 
\begin{proposition}We have the following equation for $F(u,v)$:
\begin{equation}\label{eq:baxter}
\biggl(1+\frac{tv}{1-u}+\frac{tv}{1-v/u}\biggr)F(u,v)=tuv+tuv\biggl(1+\frac{1}{1-u}\biggr)F(1,v)+\frac{tv}{1-v/u}F(u,u).
\end{equation}
\end{proposition}

\begin{proof}
In the generating tree for Baxter inversion sequences, each vertex other than the root $(1,1)$ can be generated by an unique parent. Thus, we have 
\begin{align*}
F(u,v)&=tuv+t\sum_{p,q\geq1}F_{p,q}(t)\biggl(v^{q+1}\sum_{i=1}^{p-1}u^i+uv^{q+1}+\sum_{i=0}^{q-1}u^{p+1+i}v^{q-i}\biggr)\\
&=tuv+t\sum_{p,q\geq1}F_{p,q}(t)\biggl(\frac{u-u^p}{1-u}v^{q+1}+uv^{q+1}+\frac{u^{p+q}v-u^{p}v^{q+1}}{1-v/u}\biggr)\\
&=tuv+tuv\biggl(1+\frac{1}{1-u}\biggr)F(1,v)-\frac{tv}{1-u}F(u,v)+\frac{tv}{1-v/u}(F(u,u)-F(u,v)),
\end{align*}
which is equivalent to~\eqref{eq:baxter}.
\end{proof}

Let 
$G(u,v):=\sum_{n\geq1}t^n\sum_{\pi\in\S_n(2\underline{41}3,3\underline{14}2)}u^{\lma(\pi)}v^{\rma(\pi)}$.
This formal power series $G(u,v)$ was first introduced and studied by Bousquet-M\'elou~\cite{bo}.  
Now, Theorem~\ref{thm:baxter} is equivalent to $G(u,u)=F(u,u)$, which will be established by solving~\eqref{eq:baxter}.

\begin{proof}[Proof of Theorem~\ref{thm:baxter}]
It will be convenient to set $w=v/u$ in~\eqref{eq:baxter}. The equation then becomes
$$
\biggl(1+\frac{tuw}{1-u}+\frac{tuw}{1-w}\biggr)F(u,wu)=tu^2w+tu^2w\biggl(1+\frac{1}{1-u}\biggr)F(1,wu)+\frac{tuw}{1-w}F(u,u).
$$
Further setting $u=1+x$ and $w=1+y$ in the above equation yields
\begin{multline}\label{eq:bax}
\frac{xy-t(1+x)(1+y)(x+y)}{t(1+x)(1+y)}F(1+x,(1+x)(1+y))\\
=xy(1+x)-(1-x^2)yF(1,(1+x)(1+y))-\widetilde{F}(x),
\end{multline}
where $\widetilde{F}(x):=xF(1+x,1+x)$. We call the numerator $K(x,y)$ of the coefficient of
$F(1+x,(1+x)(1+y))$ the {\em kernel} of the above equation:
$$
K(x,y)=xy-t(1+x)(1+y)(x+y).
$$
We are going to apply the so-called {\em kernel method} (cf.~\cite{bo}) to this equation. 

As a polynomial in $y$, the kernel has two roots:
$$
Y(x)=\frac{1-t(1+x)(1+\bar{x})-\sqrt{1-2t(1+x)(1+\bar{x})-t^2(1-x^2)(1-\bar{x}^2)}}{2t(1+\bar{x})},
$$
$$
Y'(x)=\frac{1-t(1+x)(1+\bar{x})+\sqrt{1-2t(1+x)(1+\bar{x})-t^2(1-x^2)(1-\bar{x}^2)}}{2t(1+\bar{x})},
$$
where $\bar{x}=1/x$. Only the first root can be substituted for $y$ in~\eqref{eq:bax}, because the term
$F(1+x,(1+x)(1+Y'))$ is not a well-defined power series in $t$ (the taylor expansion of $Y'$ in $t$
does not exist). 

Now, we will adopt the {\em obstinate kernel method} that was invented by Bousquet-M\'elou~\cite[Section~2.2]{bo} for producing all the pairs $(x,y)$ that can be legally substituted in~\eqref{eq:bax}: those are the pairs $(x,Y), (\bar{x}Y,Y), (\bar{x}Y,\bar{x})$ and their dual $(Y,x), (Y,\bar{x}Y), (\bar{x},\bar{x}Y)$, thanks to the symmetry of the kernel $K(x,y)$. 
Substituting the pairs $(x,Y)$ and $(Y,x)$ for $(x,y)$ in~\eqref{eq:bax} yields 
$$
\begin{cases}
\,\,xY(1+x)-(1-x^2)YF(1,(1+x)(1+Y))-\widetilde{F}(x)=0,
\\
\,\,xY(1+Y)-(1-Y^2)xF(1,(1+x)(1+Y))-\widetilde{F}(Y)=0.
\end{cases}
$$
Eliminating $F(1,(1+x)(1+Y))$ we get 
\begin{equation}\label{eq:1}
(x-xY^2)\widetilde{F}(x)-(Y-x^2Y)\widetilde{F}(Y)=(Y-Y^3)(x^2+x^3)-(x-x^3)(Y^2+Y^3).
\end{equation}
Similarly, substitute $(\bar{x}Y,Y),(Y,\bar{x}Y)$ and $(\bar{x}Y,\bar{x}),(\bar{x},\bar{x}Y)$ into~\eqref{eq:bax} and after some computation we get two equations, which together with~\eqref{eq:1} give the system of equations:
\begin{equation*}
\begin{cases}
\,\,(x-xY^2)\widetilde{F}(x)-(Y-x^2Y)\widetilde{F}(Y)=(Y-Y^3)(x^2+x^3)-(x-x^3)(Y^2+Y^3),
\\
\,\,(Y\bar{x}-Y^3\bar{x})\widetilde{F}(Y\bar{x})-(Y-Y^3\bar{x}^2)\widetilde{F}(Y)=(Y-Y^3)(Y^2\bar{x}^2+Y^3\bar{x}^3)-(Y\bar{x}-Y^3\bar{x}^3)(Y^2+Y^3),
\\
\,\,(Y\bar{x}-Y\bar{x}^3)\widetilde{F}(Y\bar{x})-(\bar{x}-Y^2\bar{x}^3)\widetilde{F}(\bar{x})=(\bar{x}-\bar{x}^3)(Y^2\bar{x}^2+Y^3\bar{x}^3)-(Y\bar{x}-Y^3\bar{x}^3)(\bar{x}^2+\bar{x}^3).
\end{cases}
\end{equation*}
By eliminating $\widetilde{F}(Y)$ and $\widetilde{F}(Y\bar{x})$, we get a relation between $\widetilde{F}(x)$ and $\widetilde{F}(\bar{x})$:
\begin{equation}\label{main:baxe}
\widetilde{F}(x)+\widetilde{F}(\bar{x})=\frac{Y(1+x)(x^4-2Yx^3+2Y^2x-2Y+1)}{x^2(Y-1)(Y-x)}.
\end{equation}
But $\widetilde{F}(x)=xF(1+x,1+x)$ is a formal power series in $t$ with coefficients in $x\N[x]$, while $\widetilde{F}(\bar{x})$ is a formal power series in $t$ with coefficients in $\bar{x}\N[\bar{x}]$. Therefore, the positive part in $x$ of the right hand side of~\eqref{main:baxe} is exactly $\widetilde{F}(x)$.

On the other hand, it has been shown in~\cite[Corollary~3]{bo} that if we let
$\widetilde{G}(x):=xG(1+x,1+x)$, then
\begin{equation}
\frac{x-2t(1+x)^2}{t(1+x)^2}\widetilde{G}(x)=x^2-2R(x),
\end{equation}
where $R(x)=xG(1+x,1)$.
Combining with the relation between $R(x)$ and $R(\bar{x})$ proved in~\cite[Eq.~(8)]{bo}:
$$
R(x)+R(\bar{x})=\bar{x}^2Y(1+x^3-xY),
$$
we have 
\begin{align*}\label{eq:bousquet}
\widetilde{G}(x)+\widetilde{G}(\bar{x})&=\frac{t(1+x)^2}{x-2t(1+x)^2}(x^2+\bar{x}^2-2(R(x)+R(\bar{x})))\\
&=\frac{t(1+x)^2}{x-2t(1+x)^2}(x^2+\bar{x}^2-2\bar{x}^2Y(1+x^3-xY)).
\end{align*}
To check that $\frac{t(1+x)^2}{x-2t(1+x)^2}(x^2+\bar{x}^2-2\bar{x}^2Y(1+x^3-xY))$ equals the right hand side
of~\eqref{main:baxe} is routine by Maple, which proves that $\widetilde{F}(x)=\widetilde{G}(x)$.
This completes the proof of the theorem. 
\end{proof}

Since the proof of equdistribution~\eqref{bax:equi} uses the obstinate kernel method based on the formal
power series heavily, it is natural to ask for  a bijective proof.  

\section{Euler numbers}\label{sec:euler}

The {\em Euler numbers} $E_n$ can be defined by the taylor expansion of $\tan(x)+\sec(x)$:
$$
\sum_{n\geq0} E_n\frac{x^n}{n!}=\tan(x)+\sec(x)=1+x+1\frac{x^2}{2!}+2\frac{x^3}{3!}+5\frac{x^4}{4!}+16\frac{x^5}{5!}+61\frac{x^6}{6!}+\cdots.
$$
The fundamental combinatorial interpretation of $E_n$ is due to Andr\'e~\cite{andr}, who showed that $E_n$ enumerates  permutations $\pi_1\pi_2\cdots\pi_n\in\S_n$ having the {\em down-up property}
$$
\pi_1>\pi_2<\pi_3>\pi_4<\cdots.
$$
There are several other families known to be counted by Euler numbers, including Simsum permutations and $0$-$1$-$2$-increasing trees. 

For  $\pi\in\S_n$, an index $i\geq2$ is called a {\em double descents} of $\pi$  if $\pi_{i-1}>\pi_{i}>\pi_{i+1}$.
As introduced by Simion and Sundaram~\cite{sun}, a permutation in $\S_n$ is called a
{\em Simsun permutation} if it has no double descents, even after removing $n,n-1,\ldots,k$ for any $k$. Note that Simsun permutations are slight variants of the {\em Andr\'e permutations} of Foata and Sch\"uzenberger~\cite{fo2} (see also~\cite{hei}), which were invented to interpret the  {\em$cd$-index} of symmetry groups. 
Let $RS_n$ be the set of all Simsun permutations in $\S_n$. 
Corteel et al.~\cite[Corollary~2]{cor} showed
\begin{equation}\label{simsun1}
\sum_{\pi\in RS_n} t^{\asc(\pi)}=\sum_{e\in\I_n(000)}t^{\dist(e)}
\end{equation}
via recurrence relations and posed the question of finding a natural bijection for this result. 
In this section, we will prove bijectively two different refinements of~\eqref{simsun1}.

\subsection{The Entringer--Eulerian statistics on $\I_n(000)$}
Using the statistic ``$\last$'',
we first refine~\eqref{simsun1} with a bijective proof. 
\begin{theorem}\label{euler:asc}
There exists a bijection $\Omega:\I_n(000)\rightarrow RS_n$ such that 
$$
(\dist,\last+1)(e)=(\asc,\last)\Omega(e)
$$
for each $e\in\I_n(000)$. 
Consequently, 
$$
\sum_{\pi\in RS_n} t^{\asc(\pi)}u^{\last(\pi)}=\sum_{e\in\I_n(000)}t^{\dist(e)}u^{\last(e)+1}.
$$
\end{theorem}

The bijection $\Omega$ is the combination of  the simple bijection in~\cite[Theorem~7]{cor} from $\I_n(000)$ to {\em$0$-$1$-$2$-increasing trees} with $n+1$ vertices and a special ordering of the {\em increasing tree representation} of permutations due to Maria Monks (see~\cite[Page~198]{st}).

\begin{figure}

\centering
\begin{tikzpicture}

\draw (-5,2)node[]{
\begin{tikzpicture}[every node/.style={circle,draw},level distance=1.2cm, sibling distance=2.2cm] 
\tikzstyle{level 3}=[level distance=1.2cm, sibling distance=1.2cm]
  \node {0} 
    child {node {1}
    	child {node {2}
		child {node {3}}
		child {node {4}}}
	child {node{5}
		child {node{8}}}}
    child {node {6}
      child {node {7} }}; 
\end{tikzpicture}};
\draw (-1.8,2)node{$\rightarrow$};
\draw (-1.8,2.5)node{$\mathcal{M}$};
\draw(-5,-1) node{$\uparrow$};
\draw(-5.4,-1) node{$\mathcal{C}$};
\draw(-5,-2) node{$(0,1,2,2,1,0,6,5)\in\I_8(000)$};
\draw(2,-1) node{$\downarrow$};
\draw(2.5,-1) node{$\mathcal{T}^{-1}$};
\draw(2,-2) node{$\blue{78153426}\in RS_8$};
\draw (2,1.6)node[]{
\begin{picture}(150,30)\setlength{\unitlength}{1mm}
\thinlines
\put(10,23){\circle*{2}}\put(12,23){$0(\blue{1})$}
\put(10,23){\line(-1,-1){10}}
\put(0,13){\circle*{2}}\put(2,12){$6(\blue{7})$}
\put(0,13){\line(2,-3){6}}
\put(6.5,3.5){\circle*{2}}\put(4.5,-1){$7(\blue{8})$}

\put(10,23){\line(2,-1){20}}
\put(30,13){\circle*{2}}\put(31,14){$1(\blue{2})$}
\put(30,13){\line(1,-1){10}}
\put(40,3){\circle*{2}}\put(41,4){$5(\blue{6})$}
\put(40,3){\line(1,-1){10}}
\put(50,-7){\circle*{2}}\put(51,-6){$8$(\blue{remove})}
\put(44,-3){\red{$\parallel$}}

\put(30,13){\line(-1,-1){10}}
\put(20,3){\circle*{2}}\put(21.5,1.5){$2(\blue{3})$}
\put(20,3){\line(1,-2){6}}
\put(26,-9){\circle*{2}}\put(27,-8){$3(\blue{4})$}
\put(20,3){\line(-1,-2){6}}
\put(14,-9){\circle*{2}}\put(15.5,-9){$4(\blue{5})$}
\end{picture}
};
\end{tikzpicture}

\caption{An example of the bijection $\Omega=\mathcal{T}^{-1}\circ \mathcal{M}\circ \mathcal{C}:\I_n(000)\rightarrow RS_n$.}
\label{KL:tree}
\end{figure}

It is convenient to introduce some necessary definitions about trees. 
A rooted tree with vertices labeled by a set of distinct integers is called an {\em increasing tree} if the labels of the vertices are increasing along any path from the root to a leaf.
A {\em binary increasing tree} of order $n$ is an ordered increasing tree with vertex set $\{1,2,\ldots,n\}$ in which every vertex has at most two children. In a binary increasing tree, we distinguish each child of a vertex  by left or right.  A {\em$0$-$1$-$2$-increasing tree} of order $n+1$ is an unordered  increasing tree on the vertices $\{0,1,2,\ldots,n\}$ such that every vertex has zero, one or two children. We do not consider the positions of the children of each vertex in a $0$-$1$-$2$-increasing tree. 
See Fig.~\ref{KL:tree} for a $0$-$1$-$2$-increasing tree (in left) and a binary increasing tree (in right, where we only consider the blue labels).

For $e=e_1e_2\ldots e_n\in\I_n(000)$, let $\mathcal{C}(e)$ be the unique $0$-$1$-$2$-increasing tree such that $i$ is the child of $e_i$, for all $1\leq i\leq n$. See an example of $\mathcal{C}$ when $e=(0,1,2,2,1,0,6,5)\in\I_8(000)$ in left-side of Fig.~\ref{KL:tree}. It is clear that $e\mapsto \mathcal{C}(e)$ is a bijection between $\I_n(000)$ and $0$-$1$-$2$-increasing tree of order $n+1$.   This tree representation of $000$-avoiding inversion sequences was found recently by Corteel et al.~\cite{cor}.

Let $w=w_1w_2\ldots w_n$ be a word on $\N$ with no repeated letters. Define a rooted ordered tree $\mathcal{T}(w)$ recursively as follows. If $w=\emptyset$, then $\mathcal{T}(w)=\emptyset$. Otherwise, suppose $w_i$ is the smallest letter of $w$ and define $\mathcal{T}(w)=(\mathcal{T}(w_1\ldots w_{i-1}), w_i,\mathcal{T}(w_{i+1}\ldots w_n))$, the tree with the left subtree $\mathcal{T}(w_1\ldots w_{i-1})$ and right subtree $\mathcal{T}(w_{i+1}\ldots w_n)$ attached to the root $w_i$.  The mapping $\pi\mapsto \mathcal{T}(\pi)$ is a bijection between $\S_n$ and binary increasing trees of order $n$, which is one classical tree representation of permutations (cf.~\cite[Section~1.5]{st}). For example, the tree representation of the permutation $78153426\in\S_8$ is the binary increasing tree of order $8$ in right-side of Fig.~\ref{KL:tree} (with the right-most vertex removed). 

In a binary increasing tree, the path from the root to the right-most vertex is called the {\em right-most path} of this tree.  The following result about tree representation of Simsum permutations is important.

\begin{proposition}\label{sim:tree}
A permutation $\pi$ is Simsun if and only if in $\mathcal{T}(\pi)$  the smallest child of every vertex not on the right-most path must be a right child. 
\end{proposition}
\begin{proof}
It is clear that a permutation $\pi$ has no double descent if and only if the tree $\mathcal{T}(\pi)$ has no vertex whose only child is a left child, except maybe for the rightmost vertex. The result then follows from this property and the definition of Simsun permutations. 
\end{proof}

A binary increasing tree $T$ is called a {\em Simsum tree} if $\mathcal{T}^{-1}(T)$ is a Simsum permutation. 
Given a $0$-$1$-$2$-increasing tree $T$ of order $n+1$, we can give specified position, left or right, to each child, so that 
\begin{itemize}
\item[(i)] the path from the root $0$ to $n$ moves to the right;
\item[(ii)] for every vertex which is not a leaf and not on the path from the root $0$ to $n$, its smallest child is a right child while another child (if any) becomes a left child.
\end{itemize}
We then delete the vertex $n$ (which must be on the right-most) from this ordered tree and increase each label of other vertices by one. 
In view of Proposition~\ref{sim:tree}, the resulting ordered tree, that we denote $\mathcal{M}(T)$, is a Simsun tree of order $n$. The mapping $T\mapsto \mathcal{M}(T)$ is easily seem to be a bijection between $0$-$1$-$2$-increasing trees of order $n+1$ and Simsun trees of order $n$. 
See Fig.~\ref{KL:tree}  for an example of the mapping $\mathcal{M}$.

\begin{proof}[Proof of Theorem~\ref{euler:asc}] Define $\Omega$ to be the composition $\mathcal{T}^{-1}\circ \mathcal{M}\circ \mathcal{C}:\I_n(000)\rightarrow RS_n$. 
Since $\mathcal{C}$, $\mathcal{M}$ and $\mathcal{T}$ are bijections, $\Omega$ is a bijection. See an example of $\Omega$ in Fig.~\ref{KL:tree}.  It is almost obvious from the construction that $\Omega$ transforms the pair $(\dist,\last+1)$ to $(\asc,\last)$.
\end{proof}

  \begin{remark}

It also follows from the simple tree representation $\mathcal{C}$ of $000$-avoiding inversion sequences and a result of
Poupard~\cite[Proposition~1]{pou} that the statistic ``$\last+1$'' is {\em Entrianger}, namely
$$
|\{\pi\in\Alt_{n+1}:\pi_1=k+1\}|=|\{e\in\I_n(000):\last(e)+1=k\}|,
$$
where $\Alt_{n}$ is the set of all down-up permutations in $\S_n$.
Can the generating function for this Entrianger--Eulerian pair on $\I_n(000)$ be calculated? The interested reader is referred to~\cite{fh} for the Andr\'e permutation calculus for pairs of Entrianger statistics. 
 \end{remark}
\subsection{Double Eulerian distribution on $\I_n(000)$}
In the rest of this section, we will prove a double Eulerian equidistribution (see Theorem~\ref{thm:simdou}) involving the pair $(\asc,\iasc)$ on $\I_n(000)$.
We begin with set-valued extensions of $\iasc$ and $\asc$. For each $\pi\in\S_n$, introduce the set-valued statistics
 $$\IASC(\pi):=\{\pi_i: \text{ $\pi_i$ appears on the left of $\pi_i+1$}\}$$ 
 and 
 $$
 \BOT(\pi):=\{\pi_i: \pi_i<\pi_{i+1}\}.
 $$
 We call $\BOT(\pi)$ the {\em {\bf\em bot}tom values} of the ascents of $\pi$, whose cardinality is $\asc(\pi)$.
 For example, if $\pi=78153426\in\S_8$, then $\IASC(\pi)=\{1,3,5,7\}$ and $\BOT(\pi)=\{1,2,3,7\}$.
 
 Next we construct a new coding $\Upsilon:\S_n\rightarrow\I_n$ which transforms the statistic $\BOT$ to $\ROW$. For each $\pi\in\S_n$, define $\Upsilon(\pi)=(e_1,e_2,\ldots,e_n)$, where $e_i$ equals the letter closest to $i$ in $\pi$, smaller than $i$ and left to $i$ (by convention $\pi_0=0$ is in position $0$). For example, if $\pi=78153426\in\S_8$, then $\Upsilon(\pi)=(0,1,1,3,1,2,0,7)$. 
 It is clear that if $\pi_i<\pi_{i+1}$, then $\pi_i$ in an entry of $\Upsilon(\pi)$. On the other hand, if 
 $\pi_i>\pi_{i+1}$, then $\pi_i$ is never an entry of $\Upsilon(\pi)$. Therefore, we have $\BOT(\pi)=\ROW(\Upsilon(\pi))$.
 To see that $\Upsilon$  is a bijection, we construct its inverse recursively. For each $e=(e_1,\ldots,e_n)\in\I_n$, suppose the image permutation $\pi'=\Upsilon^{-1}(e_1,\ldots,e_{n-1})$ of $(e_1,\ldots,e_{n-1})\in\I_{n-1}$ is known. Then $\Upsilon^{-1}(e)$ is obtained from $\pi'$ by inserting $n$ immediately to the right of the letter equals $e_n$ in $\pi'$. 
 
 It turns out that $\Upsilon(\pi)=V(\pi^{-1})$ for each permutation $\pi$, where  $V:\S_n\rightarrow \I_n$ is the coding named V-code in Foata~\cite{fo}. As was shown in~\cite[Th\'eor\`eme~2]{fo}, there exist another coding $S:\S_n\rightarrow \I_n$ called S-code and satisfying  
 \begin{itemize}
 \item $S(\pi)$ is word rearrangement of $V(\pi)$
 \item $\ASC(S(\pi))=\ASC(\pi)$.
 \end{itemize}
  Thus, we have the following set-valued extension of Theorem~\ref{foata}. 
 \begin{theorem}\label{set:foata}
 The bijection $\digamma:\S_n\rightarrow\I_n$ that maps $\pi$ to $S(\pi^{-1})$  has the property 
 $$
(\ROW,\ASC)\digamma(\pi)=(\BOT,\IASC)(\pi).
$$
Consequently, 
\begin{equation}\label{dou:foata2}
\sum_{\pi\in \S_n} s^{\BOT(\pi)}t^{\IASC(\pi)}=\sum_{e\in\I_n}s^{\ROW(e)}t^{\ASC(e)}
\end{equation} 
or equivalently, 
\begin{equation}\label{dou:foata}
\sum_{\pi\in\S_n} s^{\IDB(\pi)}t^{\DES(\pi)}=\sum_{e\in\I_n}s^{\ROW(e)}t^{\ASC(e)},
\end{equation} 
where $\IDB(\pi):=\{\pi^{-1}_{i+1}: \pi^{-1}_{i+1}<\pi^{-1}_i\}$ is the set of {\em{\bf i}nverse {\bf d}escent {\bf b}ottoms} of $\pi$. 
 \end{theorem}
 \begin{proof}
 By the properties of $S$-code, the bijection $\digamma$ satisfying 
  \begin{itemize}
 \item $\digamma(\pi)=S(\pi^{-1})$ is word rearrangement of $V(\pi^{-1})=\Upsilon(\pi)$
 \item $\ASC(\digamma(\pi))=\ASC(S(\pi^{-1}))=\ASC(\pi^{-1})=\IASC(\pi)$.
 \end{itemize}
 The result then follows.
 \end{proof}

 Even though the bijection $\digamma:\S_n\rightarrow\I_n$ in Theorem~\ref{set:foata} does not restrict to a bijection between $RS_n$ and $\I_n(000)$ (as $\Upsilon$ does not), we still have the following restricted version of~\eqref{dou:foata2}. 
 
\begin{theorem}  \label{thm:simdou}
There exist a bijection $\Lambda: RS_n\rightarrow\I_n(000)$
such that 
$$
(\IASC,\BOT)(\pi)=(\ASC,\ROW)\Lambda(\pi)
$$
for each $\pi\in RS_n$. 
Consequently, 
\begin{equation}\label{dou:simsun}
\sum_{\pi\in RS_n} s^{\asc(\pi)}t^{\iasc(\pi)}=\sum_{e\in\I_n(000)}s^{\dist(e)}t^{\asc(e)}.
\end{equation}
\end{theorem}

Our bijection $\Lambda$ will be a combination of S-code, $\Upsilon$-code and an intriguing bijection $\mathcal{K}$ from $0$-$1$-$2$-increasing trees of order $n+1$ to Simsun trees of order $n$, which is inspired by the {\em jeu de taquin} of Sch\"utzenberger.

\begin{figure}

\centering
\begin{tikzpicture}

\draw (-5,2)node[]{
\begin{tikzpicture}[every node/.style={circle,draw},level distance=1.2cm, sibling distance=2.2cm] 
\tikzstyle{level 3}=[level distance=1.2cm, sibling distance=1.2cm]
 \node {0} 
        child{node{2}
            child{node{5}
                child{node{7}}}
            child{node{4}}}
         child{node{1}
           child{node{3}
             child{node{8}}
             child{node{6}}}}; 
\end{tikzpicture}};
\draw (-1.8,2)node{$\rightarrow$};
\draw (-1.8,2.5)node{$\mathcal{K}$};
\draw(-5,-1) node{$\uparrow$};
\draw(-5.4,-1) node{$\mathcal{C}$};
\draw(-5,-2) node{$(0,0,1,2,2,3,5,3)\in\I_8(000)$};
\draw(2,-1) node{$\downarrow$};
\draw(2.5,-1) node{$\mathcal{T}^{-1}$};
\draw(2,-2) node{$57241386\in RS_8$};
\draw (2,1.6)node[]{
\begin{picture}(150,30)\setlength{\unitlength}{1mm}
\thinlines
\put(18,26){\circle*{2}}\put(20,26){$1$}
\put(18,26){\line(-1,-1){10}}
\put(8,16){\circle*{2}}\put(10,15){$2$}
\put(8,16){\line(-1,-2){6}}
\put(2,4){\circle*{2}}\put(-4,0){$5$}
\red{\dashline{1}(2,4)(-1,-5)\put(-1,-8){$5$}}

\put(2,4){\line(2,-3){6}}
\put(8,-5){\circle*{2}}\put(10,-6){$7$}
\red{\dashline{1}(8,-5)(15,-12)\put(16,-14){$7$}}

\put(8,16){\line(2,-3){6}}
\put(14.5,6.5){\circle*{2}}\put(12.5,2){$4$}
\red{\dashline{1}(14.5,6.5)(21.5,0.5)\put(21,-3.5){$4$}}

\put(18,26){\line(2,-1){20}}
\put(38,16){\circle*{2}}\put(39,17){$3$}
\red{\dashline{1}(38,16)(31,9)\put(28.5,6.5){$1$}}

\put(38,16){\line(1,-1){10}}
\put(48,6){\circle*{2}}\put(49,7){$6$}
\red{\dashline{1}(48,6)(55,-1)\put(55,-4.5){$6$}}

\put(48,6){\line(-1,-1){8}}
\put(40,-2){\circle*{2}}\put(42,-3.5){$8$}
\red{\dashline{1}(40,-2)(47,-9)\put(47.5,-11){$8$}}

\end{picture}
};
\end{tikzpicture}

\caption{An example of the bijection $\mathcal{K}$.}
\label{KL:jeu}
\end{figure}

Given a $0$-$1$-$2$-increasing tree $T$ of order $n+1$, remove the label $0$ of the root  and then successively move up the child with smallest label of the vertex without label. This procedure ends until the label of a leaf, say $l$, has been move up. We remove this leaf without label which results in an unordered  increasing tree on $\{1,2,\ldots,n\}$. There is a unique way to order the children of this unordered increasing tree to turn it to be a Simsun tree such that the right-most path of which is exactly the path from the root $1$ to $l$. Denote the resulting Simsun tree by $\mathcal{K}(T)$. For instance, if $T$ is the $0$-$1$-$2$-increasing tree $T$ in left-side of Fig.~\ref{KL:jeu}, then  the label $0$ is removed and the moving procedure is: (i) $1$ is moved up since $1<2$; (ii) $3$ is moved up; (iii) $6$ is moved up since $6<8$. Here $l=6$. Finally, $\mathcal{K}(T)$ becomes the Simsun tree in the right-side of Fig.~\ref{KL:jeu} (ignore the dashed line and red labels). This procedure on trees  is similar to the classical jeu de taquin on skew standard Young tableau. Since the inverse of $\mathcal{K}$ can be constructed easily, the mapping $\mathcal{K}$ is in fact a bijection between $0$-$1$-$2$-increasing tree $T$ of order $n+1$ and Simsun trees of order $n$.

\begin{lemma}\label{lem:key}
The mapping $K=\mathcal{T}^{-1}\circ \mathcal{K}\circ \mathcal{C}:\I_n(000)\rightarrow RS_n$ is a bijection satisfying 
\begin{equation}\label{bt}
\ROW(e)=\BOT(K(e))
\end{equation}
for each $e\in\I_n(000)$. Moreover, 
\begin{equation}\label{iasc}
\IASC(\Upsilon^{-1}(e))=\IASC(K(e)). 
\end{equation}
\end{lemma}

\begin{proof}
Since $\ROW(e)$ equals the set of the labels of all non-leaf and non-root vertices of $\mathcal{T}^{-1}(e)$ and $\BOT(K(e))$ equals the set of the labels of these vertices with a right child in $\mathcal{K}\circ \mathcal{C}(e)$, property~\eqref{bt} follows. Property~\eqref{iasc} is less obvious but can be proved by induction on $n$. 

Recall that $\Upsilon^{-1}(e)$ can be constructed recursively:  suppose the image permutation $\pi'=\Upsilon^{-1}(e')$ of $e'=(e_1,\ldots,e_{n-1})\in\I_{n-1}$ is known, then $\Upsilon^{-1}(e)$ can be obtained from $\pi'$ by inserting $n$ immediately to the right of the letter equals $e_n$ in $\pi'$. For example, if $e=(0,0,1,2,2,3,5,3)$ in Fig.~\ref{KL:jeu}, then $\Upsilon^{-1}(e)=2\red{5741}3\red{86}$. We call an index $0\leq i\leq n$ an {\em available inserting position} of $\Upsilon^{-1}(e)$ if $i$ appears in $e$ less than $2$ times. Let $\AVA(\Upsilon^{-1}(e))$ be the set of all available inserting positions of $\Upsilon^{-1}(e)$. We will focus on the order of the letters in $\AVA(\Upsilon^{-1}(e))$ that appear in $\Upsilon^{-1}(e)$. For our running example, the letters in $\AVA(\Upsilon^{-1}(e))=\{1,4,5,6,7,8\}$ appear in $\Upsilon^{-1}(e)$ in the order \red{$5,7,4,1,8,6$}. 

On the other hand, Simsun trees  also can be constructed recursively. Let $i$ be a vertex with less than $2$ children of a Simsun tree $T$ of order $n$. Suppose the parent of $i$ is $j$ (by convention, the parent of $1$ is $0$). We disguising four cases where we can attach $n+1$ to the vertex $i$ so that $T$ becomes a Simsun tree of order $n+1$:
\begin{itemize}
\item[(a)] If $i$ is not on the right-most path of $T$ and $i$ is a leaf, then we can attach $(n+1)$ as a right child of vertex $i$. We mark this position by $i$.
\item[(b)] If $i$ is not on the right-most path of $T$ and $i$ has a right child, then we can attach $(n+1)$ as a left child of  vertex $i$. We mark this position by $i$.
\item[(c)] If $i$ is on the right-most path of $T$ and has a right child, then we can attach $(n+1)$ as a left child of  vertex $i$. We mark this position by $j$.
\item[(d)] Otherwise, $i$ is the right-most vertex in $T$. We further disguising two cases: 
\begin{itemize}
\item[(d1)] If $i$ has a left child, then we can attach $(n+1)$ as a right child of vertex $i$ and  mark this position by $i$. 
\item[(d2)] Otherwise, $i$ is the right-most leaf of $T$. In this case, we can either attach $(n+1)$ as a left child or a right child to  vertex $i$. We mark the position in right by $i$, while the position in left by $j$. 
\end{itemize}
\end{itemize}
See Fig.~\ref{KL:jeu} (right-side) for a Simsun tree with its potential positions marked (by red integers): $i=4$, $7$ or $8$ is in case (a); $i=5$ is in case (b); $i=3$ is in case (c); $i=6$ is in case (d1).  One can check case by case that attaching  $n+1$ to a position marked $k$ in the Simsun tree $\mathcal{K}\circ \mathcal{C}(e)$ makes it become the Simsun tree $\mathcal{K}\circ \mathcal{C}(e')$, where $e'=(e_1,\ldots,e_n,k)\in\I_{n+1}(000)$. It is clear that the set of marked positions of $\mathcal{K}\circ \mathcal{C}(e)$ equals $\AVA(\Upsilon^{-1}(e))$. 

Now, property~\eqref{iasc}  is an easy consequence of the following key observation.
\vskip 0.05in

{\bf Observation:} The topological order of the marked positions of $\mathcal{K}\circ \mathcal{C}(e)$ is the same as the order (from left to right)  of the letters in $\AVA(\Upsilon^{-1}(e))$ appearing in $\Upsilon^{-1}(e)$.
\vskip 0.05in
This observation can be proved easily from the recursive constructions of $\mathcal{K}\circ \mathcal{C}$ and $\Upsilon^{-1}$ by induction on $n$, which ends the proof. 
\end{proof}

\begin{proof}[Proof of Theorem~\ref{thm:simdou}]
Let $\widetilde{RS_n}:=\{\Upsilon^{-1}(e): e\in\I_n(000)\}$.  It follows  from Lemma~\ref{lem:key} that $K\circ\Upsilon:\widetilde{RS_n}\rightarrow RS_n$ is a bijection such that
$$
 (\IASC,\BOT)(\pi)=(\IASC,\BOT)K\circ\Upsilon(\pi)
$$
for each $\pi\in\widetilde{RS_n}$. Now simply set $\Lambda=\digamma\circ\Upsilon^{-1}\circ K^{-1}$, which completes the proof in view of Theorem~\ref{set:foata}.
\end{proof}

Let 
$E_n(t):=\sum_{\pi\in RS_n}t^{\iasc(\pi)}$ be the inverse ascent polynomial on Simsun permutations. We could not find any appearance of this $t$-extension of Euler numbers in the literature. The first values of $E_n(t)$ are:
\begin{align*}
E_2(t)&=1+t,\\
E_3(t)&=4t+t^2,\\
E_4(t)&=4t+11t^2+t^3,\\
E_5(t)&=2t+32t^2+26t^3+t^4,\\
E_6(t)&=t+52t^2+161t^3+57t^4+t^5.
\end{align*}
Chow and Shiu~\cite{chow} showed that the ascent polynomials on Simsun permutations are real-rooted. It seems that this property also holds for the inverse ascent polynomials.
\begin{conjecture}\label{real:euler}
The polynomial $E_n(t)$ is real-rooted for each $n\geq2$. In particular, $E_n(t)$ is log-concave and  unimodal. 
\end{conjecture}

\section{Final remarks} 
Because of  Theorems~\ref{euler:asc} and~\ref{catalan:asc} and Conjecture~\ref{schroder:asc}, one may wonder if the same equidistribution holds for the whole sets $\S_n$ and $\I_n$ without restriction. This is in fact true as we will show in the following.
\begin{theorem}\label{thm:last}
For $n\geq1$, we have the equidistribution:
\begin{equation}\label{per:inv}
\sum_{\pi\in\S_n} t^{\asc(\pi)}u^{\last(\pi)}=\sum_{e\in\I_n}t^{\dist(e)}u^{\last(e)+1}.
\end{equation}
\end{theorem}
\begin{proof}
Since the natural coding $\Theta$ transforms the pair $(\des,n-\last)$ on $\S_n$ to $(\asc,\last)$ on $\I_n$, we have 
$$
\sum_{e\in\I_n}t^{\asc(e)}u^{\last(e)}=\sum_{\pi\in\S_n} t^{\des(\pi)}u^{n-\last(\pi)}=\sum_{\pi\in\S_n} t^{\asc(\pi)}u^{\last(\pi)-1},
$$
where the second equality follows from the simple involution on $\S_n$
$$\pi_1\pi_2\cdots\pi_n\mapsto(n+1-\pi_1)(n+1-\pi_2)\cdots(n+1-\pi_n).$$
Therefore, equidistribution~\eqref{per:inv} is equivalent to 
\begin{equation}\label{inv:last}
\sum_{e\in\I_n}t^{\asc(e)}u^{\last(e)}=\sum_{e\in\I_n}t^{\dist(e)}u^{\last(e)}.
\end{equation}

We proceed to show~\eqref{inv:last} by induction on $n$. Obviously, the result is true for $n=1$. Suppose that the result is true for $n=k$. We need to show that for a fixed $j$, $0\leq j\leq k$, 
\begin{equation*}\label{last=j}
\sum_{e\in\I_{k+1}\atop{\last(e)=j}}t^{\asc(e)}=\sum_{ e\in\I_{k+1}\atop{\last(e)=j}}t^{\dist(e)}.
\end{equation*}
Since $\sum_{e\in\I_k}t^{\asc(e)}=\sum_{e\in\I_k}t^{\dist(e)}$ by induction hypothesis, it will be sufficient to show 
\begin{equation}\label{last<j}
\sum_{e\in\I_{k}\atop\last(e)<j}t^{\asc(e)}=\sum_{e\in\I_{k}\atop{j\notin\ROW(e)}}t^{\dist(e)}.
\end{equation}

\begin{figure}
\begin{center}
\begin{tikzpicture}[scale=.45]
\draw [thick]
(0,0)--(8,0) (0,1)--(8,1) (1,2)--(8,2)(2,3)--(8,3)(3,4)--(8,4)(4,5)--(8,5)(5,6)--(8,6)(6,7)--(8,7)(7,8)--(8,8)
(0,0)--(0,1) (1,0)--(1,2)(2,0)--(2,3)(3,0)--(3,4)(4,0)--(4,5)(5,0)--(5,6)(6,0)--(6,7)(7,0)--(7,8)(8,0)--(8,8);
\draw(8.5,4.5) node{$j$};
\draw(8.5,3.2) node{$\vdots$};
\draw(8.5,1.5) node{$1$};
\draw(8.5,0.5) node{$0$};

\draw(3.5,-0.5) node{$j$};
\draw(7.5,-0.5) node{$0$};
\draw(6.5,-0.5) node{$1$};
\draw(5,-0.5) node{$\cdots$};

\draw [color=blue](4,5)--(5,4)--(6,5)--(7,4)--(8,5)(4,4)--(5,5)--(6,4)--(7,5)--(8,4);
\draw[fill=blue!25] (14,0) rectangle (15,4);
\draw [thick]
(11,0)--(19,0) (11,1)--(19,1) (12,2)--(19,2) (13,3)--(19,3) (14,4)--(19,4)(16,5)--(19,5)(17,6)--(19,6)(18,7)--(19,7)
(11,0)--(11,1) (12,0)--(12,2)(13,0)--(13,3)(14,0)--(14,4)(15,0)--(15,4)(16,0)--(16,5)(17,0)--(17,6)(18,0)--(18,7)(19,0)--(19,7);

\draw[fill=blue!25] (29,0) rectangle (30,4);
\draw [thick]
(22,0)--(30,0) (22,1)--(30,1) (23,2)--(30,2)(24,3)--(30,3)(25,4)--(30,4)(26,5)--(29,5)(27,6)--(29,6)(28,7)--(29,7)
(22,0)--(22,1)(23,0)--(23,2)(24,0)--(24,3)(25,0)--(25,4)(26,0)--(26,5)(27,0)--(27,6)(28,0)--(28,7)(29,0)--(29,7)(30,0)--(30,4);
\end{tikzpicture}
\end{center}
\caption{Three different boards.\label{board}}
\end{figure}

Now consider the three different boards in~Fig.~\ref{board}. The second board is obtained from the first board by deleting its $j$-th row, while the third board is obtained from the second one by moving the $j$-th column  to the right. By a {\em configuration} inside a board $B$, we mean a filling of  the boxes of $B$ with balls such that in each column one and only one box receives a ball. In a configuration of $B$, a row of $B$ is said to be {\em occupied} if at least one box in this row receives a ball. Note that counting the distinct entries of inversion sequences in $\{e\in\I_k:j\notin\ROW(e)\}$ is equivalent to counting the occupied rows  in configurations inside the first boards, or alternatively inside the second or third board. It then follows that 
$$
\sum_{e\in\I_{k}\atop{j\notin\ROW(e)}}t^{\dist(e)}=\sum_{e\in\I_{k}\atop\last(e)<j}t^{\dist(e)},
$$
which is equivalent to~\eqref{last<j} by the induction hypothesis. This completes the proof of the theorem by induction. 
\end{proof}

Besides Conjecture~\ref{real:euler}, the palindromic polynomial $S_n(t)$ was also conjectured in~\cite{flz} to be real-rooted. It would be interesting to investigate systematically the real-rootedness of all the Eulerian polynomials, i.e. the distribution polynomials of ascents or   distinct positive entries, on restricted inversion sequences appearing in this paper.

\section*{Acknowledgement}  This work was supported
by the National Science Foundation of China grants 11871247 and 11501244,  by the Austrian
Science Foundation FWF, START grant Y463 and SFB grant F50, by the project of Qilu Young Scholars of Shandong University, and by the National Research Foundation of Korea (NRF) grant funded by the Korea government (MSIT) (No. 2019R1F1A1062462).

\end{document}